\documentclass[reqno]{amsart}
\usepackage{color}
\usepackage{amssymb,amsmath,amsthm,amstext,amsfonts}
\usepackage{hyperref}
\usepackage{graphicx}
\usepackage{psfrag}
\usepackage{url}

\makeatletter \@addtoreset{equation}{section} \makeatother

\renewcommand\thetable{\thesection.\@arabic\c@table}

\theoremstyle{plain}
\newtheorem{maintheorem}{Theorem}

\newtheorem{theorem}{Theorem}[section]
\newtheorem{proposition}{Proposition}[section]
\newtheorem{lemma}{Lemma}[section]
\newtheorem{corollary}{Corollary}[section]
\newtheorem{definition}{Definition}[section]
\newtheorem{remark}{Remark}[section]

\newcommand{\Z}{\mathbb{Z}}
\newcommand{\N}{\mathbb{N}}
\newcommand{\R}{\mathbb{R}}

\newcommand{\dist}{\operatorname{dist}}

\newcommand{\graph}{\operatorname{graph}}
\newcommand{\Lip}{\operatorname{Lip}}

\newcommand{\cI}{{\mathcal I}}

\newcommand{\cE}{\mathcal{E}}
\newcommand{\cM}{\mathcal{M}}
\newcommand{\cB}{\mathcal{B}}

\newcommand{\cW}{\mathcal{W}}

\newcommand{\cS}{\mathcal{S}}

\hypersetup{
	colorlinks=true,
	linkcolor=blue,
	urlcolor=blue,
	citecolor=cyan,
}

\begin{document}
	\title{Sub-additivity of measure theoretic entropies of commuting transformations on Banach spaces}
	
	\author{Chiyi Luo}
	\address{Center for Dynamical Systems and Differential Equations, School of Mathematical Sciences, Soochow University\\
		Suzhou 215006, Jiangsu, P.R. China}
	\email{luochiyi98@gmail.com}
	
	\author{Yun Zhao}
	\address{Center for Dynamical Systems and Differential Equations, School of Mathematical Sciences, Soochow University\\
		Suzhou 215006, Jiangsu, P.R. China}
	\email{zhaoyun@suda.edu.cn}
	
	\thanks{This work is partially supported by NSFC (12271386, 11871361), and the second author is
		    partially supported by Qinglan project from Jiangsu Province.}
	
	\begin{abstract}
		This paper considers two commuting smooth transformations on a Banach space, and proves the sub-additivity of the measure theoretic entropies under mild conditions.  Furthermore, some additional conditions are given  for the equality of the entropies. This extends Hu's work about commuting diffeomorphisms in a finite dimensional space (Huyi Hu, 1993, Ergod. Th. Dynam. Sys., \textbf{13}: 73-100) to the case of systems on an infinite dimensional Banach space.
	\end{abstract}

	

	\maketitle

	\section{\textbf{Introduction} }\label{intro}
	In the study of dynamical systems, entropy is one of the fundamental invariant quantities that is widely used to characterize the complexity of a dynamical system, and the measure theoretic entropy measures the average amount of information and complexity in a system. For  a differentiable dynamical system on a compact finite dimensional manifold, there are several well-known results that {involve} the measure theoretic entropy, such as Ruelle's inequality, Pesin's entropy formula and {the fact} that positivity of the measure theoretic entropy implies the existence of hyperbolic periodic orbits and horseshoes. It is a natural question  whether these results in finite dimensional systems hold in infinite dimensional systems.
	
	The Oseledets multiplicative ergodic theorem (MET for short) plays a key role in the theory of differentiable dynamical systems.
    Ruelle \cite{Ruelle82} proved the MET for operator cocycles on Hilbert spaces. Thieullen \cite{Thieullen87} proved it for continuous operator cocycles on Banach spaces. For random dynamical systems, Lian and L\"{u} \cite{Lian10} proved it for {strongly} measurable operator cocycles on separable Banach spaces; Varzaneh and Riedel \cite{Varzaneh21} proved the multiplicative ergodic theorem for a strongly measurable semi-invertible operator on fields of Banach spaces, and constructed the stable and unstable manifolds theorem; and Li and Shu \cite{Li12} proved Ruelle's inequality of random dynamical systems in Banach spaces.
    
    For $C^2$  Fr\'{e}chet differentiable mappings of Banach spaces,
    Blumenthal and Young \cite{Young17} recently studied the relation between entropy and volume growth for a natural notion of volume defined on finite dimensional subspaces, and characterized SRB measures   as exactly those measures for which entropy is equal to the volume growth on unstable manifolds. This extended a significant result in \cite{Young85} for diffeomorphisms of finite dimensional Riemannian manifolds. See \cite{Lian16} for the existence of SRB measures for a $C^2$ partially hyperbolic map on a Hilbert space. Another important result in finite dimensional smooth ergodic theory is Katok's landmark approximation result that  positive measure theoretic entropy  implies the existence of hyperbolic periodic orbits and horseshoes \cite{Katok80}. Lian and Young \cite{Lian11} generalized Katok's theorem to Fr\'{e}chet differentiable maps on {a} Hilbert space  with injective, compact derivative and admitting a compact invariant set on which the invariant measure is supported, and later to a similar setting for a $C^2$ smooth semiflow on Hilbert spaces \cite{Lian12}. Later, Lian and Ma \cite{Lian20} extended this result to maps in a separable Banach space. These previous researches {were} driven by the study of dissipative parabolic differential equations having global attractors, for instance, 2D Navier-Stokes equations and other reaction diffusion equations.

    If $f,g$ are two commuting $C^{2}$ diffeomorphisms on a smooth compact manifold, Hu \cite{Hu96} proved the sub-additivity of measure theoretic entropies, i.e., $h_{\mu}(f\circ g)\leq h_{\mu}(f)+h_{\mu}(g)$ with respect to any $(f,g)$-invariant measure $\mu$ (see Section \ref{setting} for the definition). This result does not remain true for measure preserving transformations in general,  see an example in the end of \cite{Hu96}, where the author constructed two commuting measure preserving transformations $f,g$ of a compact manifold $M$ satisfying that  $h_{\mu}(f)=h_{\mu}(g)=0$ and $h_{\mu}(fg)>0$ with respect to a measure $\mu$.
    The main aim of this paper is to generalize the above result to  $C^{2}$ Fr\'{e}chet differentiable mappings of Banach spaces. Namely, let $f,g$ be two commuting $C^{2}$ Fr\'{e}chet differentiable maps on a Banach space $\mathcal B$, we will show that $h_{\mu}(f\circ g)\leq h_{\mu}(f)+h_{\mu}(g)$ with respect to an $(f,g)$-ergodic measure $\mu$ under some suitable conditions.

	Although one may utilize some existing techniques to deal with the lack of local compactness of the phase space, there are some new questions {that} should be addressed in our case of the infinite dimensional Banach spaces. The first one is  caused by the non-invertibility of the maps $f$ and $g$. The maps may {not} be onto and {may have} arbitrarily strong contraction in some directions. Even the map $f^{-1}$ can be defined in some invariant subset, one cannot expect it to have nice regular properties. The second is that the Lyapunov metric does  depend on points, one has difficulties in  estimating  the norm of vectors since there is no inner product.
	
	This paper is organized as follows: In Section \ref{setting}, we give some preliminary results and the statements of the main results.  In Section \ref{MET} and Section \ref{SEC4}, we recall the Multiplicative Ergodic Theorem for commuting transformations and  establish the Lyapunov norms in our case. In Section \ref{S5} we give the {details} of  the sub-additivity of entropies. In Section \ref{S6} and Section \ref{S7} we provide some properties of unstable manifolds and give a condition of the equality of the entropies in Section \ref{S7}.
	
	\section{\textbf{Preliminaries  and statements}}\label{setting}
	Throughout this paper, let $\cB$ denote {an} (infinite dimensional) separable Banach space with norm $|\cdot|$, and let $A\subset \cB$ be a compact subset. 
	Consider two maps $f:\cB \rightarrow \cB$ and $g:\cB \rightarrow \cB$, and denote by $fg$ the composition $f\circ g$ for simplicity,  assume the following properties are satisfied:
	\begin{enumerate}
		\item[(H1)] 
		\begin{enumerate} 
			\item[(i)]  $fg=g f$, both $ f $ and $ g $ are $C^2$ Fr\'{e}chet differentiable and injective;
			\item[(ii)] denoted by $Df_x$ and $Dg_x$ the derivative of $f$ and $g$ at the point $x\in \cB$, they are also injective.
		\end{enumerate}
		\item[(H2)] 
		\begin{enumerate} 
			\item[(i)]  $ f(A)=A $ and $ g(A)=A $;  	
			\item[(ii)] the set $A$ has finite (upper) box-counting dimension.
		\end{enumerate}
		\item[(H3)] Assume that
		$$ l_{\alpha}(f)=\lim_{n\rightarrow \infty}\frac{1}{n} \log \sup_{x\in A} |Df^{n}_x|_{\alpha}<0 \ \text{,} \
		   l_{\alpha}(g)=\lim_{n\rightarrow \infty}\frac{1}{n} \log \sup_{x\in A} |Dg^{n}_x|_{\alpha}<0.   $$
		Here $|T|_{\alpha}$ is the Kuratowski measure of noncompactness for a linear operator $ T $ (see definition in below).
	\end{enumerate}
	
	\begin{definition}
		The Kuratowski measure of noncompactness for a linear operator $ T : \cB\rightarrow \cB$ is defined as
		$$|T|_{\alpha}=\inf_{r>0}\Big\{r:T(B) \ \text{can be covered by finite number of balls with radius}\ r \Big\}$$
		where $B$ is the unit ball in $ \cB $.
	\end{definition}
	Let $h=f \ \text{or} \ g$, since $|T_1  T_2|_{\alpha}\leq |T_1|_{\alpha} \cdot |T_2|_{\alpha}$ for any linear operators $T_1,T_2$ on $\cB$ (e.g., see \cite[Section 2]{Lian10} ),
	we have that
	$$\sup_{x\in A} |Dh^{n+k}_x|_{\alpha}\leq \sup_{x\in A} |Dh^{n}_x|_{\alpha} \cdot \sup_{x\in A} |Dh^{k}_x|_{\alpha}.$$
	Hence, all the limits in (H3) exist.
	\begin{remark}
		The conditions (H1) and (i) of (H2) are natural properties  guaranteeing the existence of Lyapunov exponents in the Oseledets multiplicative ergodic theorem on a Banach space (see \cite{Thieullen87}). 
		Condition (H3) is discussed in Section \ref{SEC4}, it implies that the tangent space {can} be decomposed  into the unstable, center and stable subspaces with respect to the transformation. 
		For (ii) of (H2) , it is naturally satisfied when $Df_x$(or $Dg_x$) is the sum of a compact operator and a contraction for each $x\in A$ (see \cite{Mane81}), and it will be used in Sections \ref{sub5.3} and \ref{sub7.2}.
	\end{remark}
	
	For every measurable transformation $h:\cB \rightarrow \cB$, let $\cM(h,A)$ denote the set of all $h$-invariant Borel probability measures supported on $A$, and put $\cM(f,g,A):=\cM(f,A) \cap \cM(g,A)$. It is proved in \cite{Hu96} that the set $\cM(f,g,A)$ is nonempty. A measure $\mu \in \cM(f,g,A)$ is $(f,g)$-ergodic, if for any measurable set $B$ with $\mu(f^{-1}(B) \triangle B)=\mu(g^{-1}(B) \triangle B)=0$, one has $\mu(B)=0$ or $1$. Let $\cE(f,g,A)$ denote the set of all $(f,g)$-ergodic measures. Note that if $\mu \in \cM(f,g,A)$ is $f$-ergodic or $g$-ergodic then $\mu \in \cE(f,g,A)$, and  the inverse is not true in general.
	
	We recall some basic properties of the sets $\cM(f,g,A)$ and $\cE(f,g,A)$, and refer the reader to \cite[Section 1]{Hu96} for more details.
	\begin{proposition}\label{P2}
		If $T$ and $S$ are commuting continuous maps on a compact metric space $X$, then
		\begin{enumerate}
			\item[(1)]
			\begin{enumerate}
				\item[(i)]	 $T$ and $S$ have common invariant measures, i.e., $\cM(T,S,X)\neq \varnothing$;
				\item[(ii)]	 the set $\cM(T,S,X)$ is convex;
				\item[(iii)] $\cM(TS,S,X)=\cM(T,S,X)$.
			\end{enumerate}
			\item[(2)]
			\begin{enumerate}
				\item[(i)]   $\cE(T,S,X)$ is the set of all extreme point of $\cM(T,S,X)$;
				\item[(ii)]  for every $\mu\in \cM(T,S,X)$, there exists a unique Borel probability measure $\pi$ on $\cE(T,S,X)$
				such that $$\mu=\int_{\cE(T,S,X)} v_e\, d\pi(v_e)$$
				\item[(iii)] $\cE(TS,S,X)=\cE(T,S,X)$.
				\item[(iv)]  if $\mu \in \cE(T,S,X)$, then for any measurable function $\psi$ on $X$ with $\psi(f(x))\leq \psi(x)$ and $\psi(g(x))\leq \psi(x)$, one has that $\psi$ is a constant for $\mu$-a.e. $x$.
			\end{enumerate}	
		\end{enumerate}
	\end{proposition}
	
	Let $h=f,g$ or $fg$. Denote by $h_\mu(h)$ the measure theoretic entropy of $h$ with respect to $\mu$ (see \cite[Chapter 4]{Walters82} for details).
	The assumption (i) of (H1) implies that $h|_{A}$ is invertible, and $h_\mu(h)=h_\mu (h|_{A})$ for every $\mu\in \cM(f,g,A)$.
	
	\begin{definition}
		A sub-bundle $V$ is  \textbf{finite dimensional $Dh$-invariant}  with respect to a measure  $\mu \in \cM(f,g,A)$ if it satisfies the following properties:
		\begin{enumerate}
			\item[(1)] $x\mapsto V(x)$ is measurable and $\dim V(x)$ is finite for $\mu$-a.e. $x$;
			\item[(2)] $Dh_x V(x)=V(h(x))$ for $\mu$-a.e $x$,
		\end{enumerate}
		where $h=f$ or $g$.
	\end{definition}
	
	Given $\mu \in \cM(f,g,A)$, assume that the following condition is satisfied:
	\begin{enumerate}
		\item[(H4)] For every finite dimensional sub-bundle $V$ which is both $Df$-invariant and $Dg$-invariant with respect to the measure $\mu\in \cM(f,g,A)$, then either $$\int \log^{+}|(Df_x|_{V(x)})^{-1}|d\mu(x)<\infty \ \text{and} \ \int\log^{+}|(Dg_x|_{V(x)})^{-1}|d\mu(x)<\infty$$
        or both $\log^{+}|(Df|_{V})^{-1}|$ and $\log^{+}|(Dg|_{V})^{-1}|$ are not integrable  with respect to $\mu$.
	\end{enumerate}
	\begin{remark}
    The above condition is satisfied for every $\mu\in \cM(f,g,A)$ if $f$ and $g$ are $C^2$ diffeomorphisms on some neighborhood of $A$. It is used in the proof of Corollary \ref{C3} in Section \ref{MET}, and which is only used in the proof of (e) of Theorem \ref{MET fg}.
	\end{remark}

	Under the previous assumptions, we have the following theorem.
	
	\begin{maintheorem}\label{A}
		Let $\cB$ be a separable Banach space with norm $|\cdot|$. Given a compact subset $A\subset \cB$ and  a measure $\mu \in \cE(f,g,A)$. If $f,g:\cB \rightarrow \cB$ satisfy the conditions (H1)-(H4), then we have that
		$$ h_\mu(f g)\leq h_{\mu}(f)+h_{\mu}(g).$$
	\end{maintheorem}

    When does the inequality in the above theorem become an equality? To answer this question, we introduce some notations first. 	
	Let $\cB_x:=\{x\}+\cB$ denote the tangent space of $\cB$ at $x\in \cB$ and choose $\lambda_{\alpha}\in \R$ such that
	\begin{equation}\label{eq:lammda-a}
		0>\lambda_{\alpha}>\max\{l_{\alpha}(f),l_{\alpha}(g)\}.
	\end{equation}
    Let $h=f$ or $g$. Given $\mu\in \cE(f,g,A)$, denote by $\lambda_1(x,h)>\lambda_2(x,h)>\cdots>\lambda_{r(x,h)}(x,h)>\lambda_{\alpha}$ the distinct Lyapunov exponents of $(h,\mu)$  which are bigger than $\lambda_{\alpha}$ for $\mu$-almost every $x$.
    Define respectively the unstable, center and stable subspace of $\cB_x$ by
	$$E^{u}(x,h)=\bigoplus_{\{i:\lambda_i(x,h)>0\}}E_i(x,h) \ \text{,} \ E^{c}(x,h)=\bigoplus_{\{i:\lambda_i(x,h)=0\}}E_i(x,h)$$
	and $$E^{s}(x,h)=\bigoplus_{\{i:\lambda_i(x,h)<0\}}E_i(x,h)\oplus E_{\alpha}(x,h).$$
    See Theorem \ref{MET h} for detailed description of the above notations.
	For any $\delta>0$, denote by $w^{u}_{\delta}(x,h)$ the local unstable manifold of $h$ at the point $x$ (see Section \ref{S6} for the precise  definition).
	
	\begin{maintheorem}\label{B}
		Let $\cB$ be a separable Banach space with norm $|\cdot|$. Given a compact subset $A\subset \cB$ and  a measure $\mu \in \cE(f,g,A)$. Assume that $f,g:\cB \rightarrow \cB$ satisfy the conditions (H1)-(H4) and $E^{u}(x,f)=E^{u}(x,g)$ for $\mu$-almost every $x$.
		Then, there exists a number $\hat{\delta}>0$ such that
		$$w^{u}_{\delta}(x,f)=w^{u}_{\delta}(x,g)\,\, \mu-a.e. \, x $$
		for every $0<\delta<\hat{\delta}$. Moreover, if $E^{c}(x,f)=\{0\}$ {and}  $E^{c}(x,g)=\{0\}$  for $\mu$-almost every $x$, then
		$$h_\mu(fg)=h_{\mu}(f)+h_{\mu}(g).$$
	\end{maintheorem}
	
	\section{\textbf{ MET for commuting transformations}} \label{MET}
	In this section, we will construct an Oseledets splitting of the tangent space $\cB_x$ such that the Lyapunov exponents of $Df,Dg$ and $D(fg)$ are well-defined on the finite dimensional subspace, and can also be controlled on the finite co-dimensional subspace.
	
	\begin{definition}\label{D1}
		Let $X$ be a compact metric space and $\mu$  a Borel probability measure on $X$, and let $Z$ be a metric space. We say that a map $\Psi:X\rightarrow Z$ is $\mu$-continuous, if there is an increasing sequence $\{K_{n}\}_{n\in \N}$ of compact subsets of $X$, such that $\mu(\cup_{n}K_n)=1$ and $\Psi|_{K_n}$ is continuous for each $n$.
	\end{definition}
	
	Given $f,g,\cB,A$ satisfying the conditions (H1)-(H3) and  $\mu\in \cE(f,g,A)$. Put $h=f$ or $g$. By the sub-additive ergodic theorem, the following limit exists for $\mu$-almost every $x$:
	$$l_{\alpha}(x,h):=\lim_{n\rightarrow \infty}\frac{1}{n} \log |Dh^{n}_x|_{\alpha}.$$
	It follows from (H3) immediately that $l_{\alpha}(x,h)\leq l_{\alpha}(h)<\lambda_{\alpha}$ (recall the choice of $\lambda_{\alpha}$ in \eqref{eq:lammda-a}). We first recall the classical MET  for derivative cocycles on Banach space, see \cite{Thieullen87} for more details.
	
	\begin{theorem}\label{MET h}
		Let $\cB,h,A,\lambda_{\alpha}$ and $\mu$ be as in above. Then there exists a both $f$-invariant and $g$-invariant Borel subset $\Gamma$ with $\mu(\Gamma)=1$ such that for every $x\in \Gamma$, there are $r(x,h)$ numbers
		$$\lambda_1(x,h)>\lambda_2(x,h)>\cdots>\lambda_{r(x,h)}(x,h)>\lambda_{\alpha}$$
		and a unique splitting
		$$\cB_x=E_1(x,h) \bigoplus \cdots \bigoplus E_{r(x,h)}(x,h) \bigoplus E_{\alpha}(x,h)$$
		with the following properties:
		\begin{enumerate}
			\item[(a)] for $i=1,\cdots,r(x,h)$, $ \dim E_{i}(x,h)=m_i(x,h)$ is finite and $Dh_x E_{i}(x,h)=E_{i}(h(x),h)$, for every $v\in E_{i}(x,h)\setminus\{0\}$ we have
			$$
			\lim_{n\rightarrow \pm \infty} \frac{1}{n}\log |Dh^{n}_x v|=\lambda_{i}(x,h);
			$$
			\item[(b)] $E_{\alpha}(x,h)$ is closed and finite co-dimensional, satisfies that
			$Dh_x E_{\alpha}(x,h)$ $\subset E_{\alpha}(h(x),h)$ and
			$$ \limsup_{n\rightarrow \infty} \frac{1}{n}\log |Dh^{n}_x|_{E_{\alpha}(x,h)}|\leq \lambda_{\alpha}; $$
			\item[(c)] the map $x\mapsto E_{i}(x,h)$ ($\ i=1,\cdots,r(x,h)$, $\alpha$) is $\mu$-continuous;
			\item[(d)] let $\pi_{i}(x,h)$ ( $i=1,\cdots,r(x,h)$, $\alpha$) denote  the projection of $\cB_x $ onto $ E_{i}(x,h) $ via the splitting, we have
			$$\lim_{n\rightarrow \pm\infty} \frac{1}{n}\log |\pi_{i}(h^{n}(x),h)|=0.$$
		\end{enumerate}
	\end{theorem}
	\begin{remark}
		If $\mu$ is $h$-ergodic, it is well-known that $r(x,h)$, $\lambda_i(x,h)$  and $m_i(x,h)$ ($i=1,\cdots,r(x,h)$) are constant $\mu$-almost everywhere. In Section \ref{setting}, we stressed that $\mu\in \cE(f,g,A)$ does not imply $\mu$ is $f$-ergodic or $g$-ergodic. But we will show that all of these functions are constants $\mu$-almost everywhere if  $\mu$ is only $(f,g)$-ergodic.
	\end{remark}
	
	\begin{remark}
		Note that for every metric space $Z$, $\mu$-continuous is equivalent with Borel measurability (see \cite[Theorem 4.1]{Kupka84}).
	\end{remark}
	
	For a given measurable positive function $\psi$ and a measure $\mu$, we call $\psi$ is ($\mu$-)\textbf{temperate} with respect to $h$, if
	$$\lim_{n\rightarrow \pm \infty} \frac{1}{n}\log \psi\circ h^n=0 \quad \mu\text{-}a.e.$$
	and we call $\psi$ is $\varepsilon$-\textbf{slowly} with respect to $h$, if
	$$\psi(h^{\pm}x)\leq e^{\varepsilon}\psi(x)\quad \mu\text{-}a.e. $$
	where $\varepsilon>0$ and $h=f$ or $g$.
	
	The following lemmas are useful in the proof of our main results, see Lemmas 8 and 9 in \cite{Walters93} for the first result, and Lemma 5.4 in \cite{Young17} for the second result.

	\begin{lemma} \label{Tem}
		Let $T: X\to X $ be an invertible measure preserving transformation of a probability space $(X,\mu)$, and let $\psi:X\rightarrow \R$ be a measurable function. Assume that either $(\psi\circ T-\psi)^{+}$ or $(\psi\circ T-\psi)^{-}$ is integrable, then
		$$\lim_{n\rightarrow \pm \infty} \frac{1}{n} \psi\circ T^n(x)=0 \quad \mu\text{-}a.e. \ x\in X.$$
	\end{lemma}
	\begin{lemma}\label{slowly}
    Let $T: X\to X $ be an invertible measure preserving transformation of a probability space $(X,\mu)$, and let $\psi:X\rightarrow \R$ be a measurable and temperate function.
		  Then, for every $\varepsilon>0$, there exists a  measurable function ${\psi}':X\rightarrow [1,\infty)$ such that
		${\psi}'(x)\geq \psi(x)$ and
		${\psi}'(T^{\pm}x)\leq e^{\varepsilon}{\psi}'(x)$ for $\mu$-almost every  $x\in X$.
	\end{lemma}

    \subsection{Invariant splitting}
    We first prove the Oseledet splitting of $f$ is also $Dg$-invariant, and so it is both $Df$-invariant and $Dg$-invariant.
	Let $ m(x,f)=m_1(x,f)+\cdots+m_{r(x,f)}(x,f)$ and
	$$E(x,f)=E_1(x,f) \bigoplus \cdots  \bigoplus E_{r(x,f)}(x,f).$$

	\begin{lemma}\label{L1}
		For every $x\in \Gamma$ and $v\in E(x,f)\setminus\{0\}$, let
		$$\lambda(x,v,f)=\lim_{n\rightarrow \infty}\frac{1}{n} \log|Df^n_xv|$$
		if the limit exists. Then we have $\lambda(g(x),Dg_xv,f)=\lambda(x,v,f)$. 
	\end{lemma}
	
	\begin{proof}
		Let $C_0=\max\{\sup_{x\in A}|Df_x|,\sup_{x\in A}|Dg_x|\}<\infty$, and put
		$$m(x,g,E):=m(Dg_x|_{E(x,f)})=\inf_{v\in E(x,f)\setminus\{0\}}\dfrac{|Dg_x v|}{|v|}$$
		then one has that $m(x,g,E)>0$ for every $x\in \Gamma$,  since $Dg_x$ is injective and $E(x,f)$ is finite dimensional.
		
		{\bf We claim that for $\mu$-almost every $x\in \Gamma$
		\begin{equation}\label{c-lz}
		\lim_{n\rightarrow \pm \infty} \frac{1}{n} \log m(f^{n}(x),g,E)=0.
		\end{equation}}
		
     We proceed to prove the lemma by assuming that the claim is true. By definition, we have $m(x,g,E)|v|\leq |Dg_x v| \leq C_0|v|$ for every $x\in\Gamma$ and $v\in E(x,f)$, this together with the community of $f,g$ yield that
		\begin{align*}
		m(f^{n}(x),g,E)|Df^n_x v| \leq |Dg_{f^{n}(x)} Df^n_x v|=|Df^n_{g(x)}Dg_x v|\leq C_0 |Df^n_x v|
		\end{align*}
	 It follows from the claim immediately that $\lambda(g(x),Dg_xv,f)=\lambda(x,v,f)$ for $\mu$-almost every $x\in \Gamma$.

     To finish the proof of the lemma, it suffices to prove the above claim.

     \noindent	{\it Proof of the claim:} Let $\Gamma(m)=\{x\in \Gamma: m(x,f)=m\}$, then $\Gamma=\bigcup_{m\in \N} \Gamma(m)$ and $\Gamma(m)$ is  $f$-invariant  for every $m\in \N$. Take a positive integer $m\in \N$ with $\mu(\Gamma(m))>0$, let $\mu_m$ be the normalized measure of $\mu$ restricted on $\Gamma(m)$. Hence, $\mu_m$ is an $f$-invariant probability measure supported on $\Gamma(m)$.
		Note that
		$$Dg_{f(x)}|_{E(f(x),f)}=Df_{g(x)} \circ Dg_{x} \circ (Df_x |_{E(x,f)})^{-1},$$
		it is easy to show that
		$m(f(x),g,E)\leq |Df_{g(x)}|\cdot m(x,g,E) \cdot |(Df_x |_{E(x,f)})^{-1}|$. Consequently, one can show that
		\begin{eqnarray}\label{c-ad1}
        (\log m(f(x),g,E)-\log m(x,g,E))^{+}\leq \log C_0 + \log^{+} |(Df_x |_{E(x,f)})^{-1}|.
        \end{eqnarray}
		To complete the proof of the lemma, it suffices to show that \eqref{c-lz} holds for $\mu$-a.e. $x\in \Gamma(m)$. By Lemma \ref{Tem} and \eqref{c-ad1}, it suffices to show that
		\begin{eqnarray}\label{c-ad2}
        \int \log^{+} |(Df_x |_{E(x,f)})^{-1}| d\mu_m(x)<\infty.
        \end{eqnarray}
		For every $x\in \Gamma(m) $ and every $v\in E(x,f)$ with $|v|=1$,  by \cite[Lemma 3.11]{Young17} we have  that
		\begin{eqnarray}\label{c-lzz}
        \det(Df_x |_{E(x,f)}) \leq m^{m/2} |Df_x |_{E(x,f)} |^{m-1} |Df_x v|
        \end{eqnarray}
        one can see \cite[Section 3.2]{Young17}  for the definition of the determinant of a linear operator on Banach spaces. This together with the fact $ |Df_x |_{E(x,f)} |\leq C_0$ yield that
		$$|(Df_x |_{E(x,f)})^{-1}|\leq C_m \cdot (\det(Df_x |_{E(x,f)}))^{-1}$$
		where $C_m=m^{m/2} \cdot C_0^{m-1}$.

        Put $\phi(x)=\log \det(Df_x |_{E(x,f)})$. To show \eqref{c-ad2}, it suffices to prove that $ \phi\in L^{1}(\mu_m)$. Since $\det(Df_x |_{E(x,f)})\leq m^{m/2} C^{m}_0$ by \eqref{c-lzz}, one has $\phi^{+}\in L^1(\mu_m)$. It remains to show that $\phi^{-}\in L^1(\mu_m)$. By the Birkhoff's ergodic theorem, there exists a measurable function $\hat{\phi}:\Gamma_m\rightarrow [-\infty,\infty)$ with $\hat{\phi}^{+}\in L^1(\mu_m)$ such that
		$$ \hat{\phi}(x)=\lim_{n\rightarrow \infty} \frac{1}{n}\sum_{k=0}^{n-1} \phi(f^i(x))=\lim_{n\rightarrow \infty} \frac{1}{n}\sum_{k=0}^{n-1} \phi(f^{-i}(x))$$
		for $\mu_m$-almost every $x\in \Gamma(m)$. Note that
		\begin{align*}
		\hat{\phi}(x)&=\lim_{n\rightarrow \infty} \frac{1}{n}\sum_{k=0}^{n-1} \phi(f^{-i}(x))\\
                     &=\lim_{n\rightarrow \infty}  \frac{1}{n} \sum_{k=0}^{n-1}\log \det(Df_{f^{-i}(x)} |_{E(f^{-i}(x),f)})\\
		             &=\lim_{n\rightarrow \infty} \frac{1}{n} \log \det(Df^n_{f^{-n}(x)}|_{E(f^{-n}(x),f)})\\
		             &=\lim_{n\rightarrow \infty} -\frac{1}{n} \log \det(Df^{-n}_{x}|_{E(x,f)})
		\end{align*}
		and, by \cite[Lemma 3.11]{Young17} one can choose $\{v_1,\cdots,v_m \}$ to be an unit basis for $E(x,f)$ such that
		$$\det(Df^{-n}_{x}|_{E(x,f)})\leq m^{m/2} \prod_{i=1}^{m}|Df^{-n}_{x} v_i|.$$
		By Theorem \ref{MET h} (a), we have that
		\begin{align*}
		\hat{\phi}(x)&\geq \lim_{n\rightarrow \infty} -\frac{1}{n} \log \prod_{i=1}^{m}|Df^{-n}_{x} v_i|\\
		             &=\sum_{i=1}^{m} \lim_{n\rightarrow \infty} -\frac{1}{n} \log |Df^{-n}_{x} v_i|>  m\lambda_{\alpha}>-\infty.
		\end{align*}
	    Then, we have that $\int \hat{\phi}(x) dx>-\infty$.
		Note that $\int \hat{\phi}(x) dx= \int \phi(x) dx$. Hence, we have $\phi(x) \in L^1(\mu_m).$

        This completes the proof of this lemma.
	\end{proof}
	
	Actually, the following result is proved in the proof of the above claim.
	\begin{lemma}\label{L1.5}
		Given a finite dimensional $Dh$-invariant sub-bundle $V$ (here $h$ is $f$ or $g$), if there exists a number $\lambda>-\infty$ such that for $\mu$-almost every $x$
		$$\lim_{n\rightarrow \pm \infty} \frac{1}{n} \log |Dh^{n}_{x} v|\geq \lambda$$
		for any $v\in V(x)\setminus\{0\}$. Then, we have $\int \log^{+} |(Dh_x|_{V(x)})^{-1}|d\mu<\infty$.
	\end{lemma}

	\begin{lemma}\label{L2}
		For $\mu$-almost every $x\in \Gamma$, we have that $$Dg_x E_{j}(x,f)=E_{j}(g(x),f)\ \ j=1,\cdots,r(x,f),$$  and the functions $r(\cdot,f),\ \lambda_{j}(\cdot,f)$ and $m_j(\cdot,f)$ ( $j=1,\cdots,r(\cdot,f)$) are constants $\mu$-almost everywhere.
	\end{lemma}
	\begin{proof}
		By Lemma \ref{L1}, for $\mu$-almost every $x\in \Gamma$ and every $u\in E(x,f)\setminus \{0\}$ we have that
		$$\lim_{n\rightarrow \infty}\frac{1}{n} \log |Df^n_{g(x)}Dg_x u|\geq \lambda_{r(x,f)}(x,f)>\lambda_{\alpha}.$$
		This implies that $Dg_x u \notin E_{\alpha}(g(x),f)$, that is, $E_{\alpha}(g(x),f)\cap Dg_x E(x,f)=\{0\}$. Hence, one has that
		$$m(x,f)\leq \text{co}\dim E_{\alpha}(g(x),f)=m(g(x),f)$$
		for $\mu$-almost every $x\in \Gamma$. Since $\mu$ is $(f,g)$-ergodic and $m(x,f)=m(f(x),f)$, by Proposition \ref{P2} we have that $m(\cdot,f)$ is a constant $\mu$-almost everywhere.
		
		Take a point $x$ with $m(x,f)=m(g(x),f)$,  one can show that $\lambda_{j}(x,f)\le\lambda_{j}(g(x),f)$ by Lemma \ref{L1} and, $m_j(x,f)\le m_{j}(g(x),f)$ by the injectivity of $Dg_x$ for $j=1,\cdots,r(x,f)$. Consequently, one has that $r(x,f)\le r(g(x),f)$ since we have already obtained $r(x,f)$ number of exponents at point $g(x)$. Note that $r(\cdot,f),\ \lambda_{j}(\cdot,f)$ and $m_j(\cdot,f)$  are also $f$-invariant functions, it follows from the $(f,g)$-ergodity that  they are constants for $\mu$-almost every $x$.

        Since $E_{\alpha}(g(x),f)\cap Dg_x E(x,f)=\{0\}$ and $Dg_x$ is injective for $\mu$-almost every $x$, one may consider a splitting of the tangent space at point $x$ as follows:
        $$
        \cB_x=\bigoplus_{i=1}^{r}Dg_{g^{-1}x}E_i(g^{-1}(x),f)\bigoplus E_{\alpha}(g(x),f).
        $$
        It satisfies Theorem \ref{MET h}. By the unique existence of Oseledets splitting, we have that $Dg_x E_{j}(x,f)=E_{j}(g(x),f)$ for $j=1,\cdots,r(x,f)$ and $Dg_x E_{\alpha}(x,f)\subset E_{\alpha}(g(x),f)$
        for $\mu$-almost every $x$.
	\end{proof}
	One can prove in a similar fashion  that $Df_x E_{j}(x,g)=E_{j}(f(x),g)$ and the functions $r(\cdot,g)$, $\lambda_{j}(\cdot,g)$ and $m_j(\cdot,g)$, $j=1,\cdots,r(\cdot,g)$ are constants $\mu$-almost everywhere. Thus, for $h=f$ or $g$, we write $r(x,h)=r(h)$ and  $\lambda_{j}(x,h)=\lambda_{j}(h),\ m_j(x,h)=m_j(h)$ ( $j=1,\cdots,r(h)$) for $\mu$-almost every $x$.
	
	\subsection{Construction of a new splitting} In this section, we will construct a desired splitting of the tangent space $\cB_x$ for the purpose of proving our main result. We first introduce a concept in order to simplify our description.
	 Given a finite dimensional both $Df$-invariant and $Dg$-invariant sub-bundle $V$, and let $h=f,g$ or $f g$, we call the number $\lambda(h)\in \R$  \textbf{the Lyapunov exponent of $Dh$ on sub-bundle $V$}, if
	$$\lim_{n\rightarrow \pm \infty} \frac{1}{n}\log |Dh^{n}_x v|=\lambda(h)\quad  \mu-a.e.\, x$$
	for every $v\in V(x)\setminus \{0\}$.
	\begin{lemma}\label{L4}
		 Given a finite dimensional both $Df$-invariant and $Dg$-invariant sub-bundle $V$. If there exist two numbers $\lambda(f)\in \R$ and $\lambda(g)\in \R$ such that for $\mu$-almost every $x$
		$$\lim_{n\rightarrow \pm \infty} \frac{1}{n}\log |Df^{n}_x v|=\lambda(f) \ \text{and} \
		\lim_{n\rightarrow \pm \infty} \frac{1}{n}\log |Dg^{n}_x v|=\lambda(g)
		$$ for every $v\in V(x)\setminus \{0\}$. Then, for every $s,t\in \Z$ and $\mu$-almost every $x$ we have that
		$$\lim_{n\rightarrow \pm \infty} \frac{1}{n}\log |D(f^sg^t)^{n}_x v|=s\lambda(f)+t\lambda(g)\quad \forall v\in V(x)\setminus \{0\}. $$
			\end{lemma}
	\begin{proof}
		Note that $f,g$ are $C^2$ maps, by the assumption and Lemma \ref{L1.5} we have that
		$$\log^{+} |(Df|_{V})^{\pm}|\in L^{1}(\mu) \ \text{and} \ \log^{+} |(Dg|_{V})^{\pm}|\in L^{1}(\mu).$$
		This implies that  $\log^{+} |Df^sg^t|_{V}|\in L^{1}(\mu)$ for every $s,t\in \Z$.
		For every $s,t\in \Z$, consider the finite dimensional $D(f^sg^t)$-invariant sub-bundle $V$, and apply the multiplicative ergodic theorem for the cocycle $\{D(f^sg^t)^{n}\}_{n\ge 0}$ which is defined on $A$ with respect to $f^sg^t$,  for $\mu$-almost every $x\in \Gamma$ the following limit
		$$\lambda(x,v,f^sg^t)=\lim_{n\rightarrow \pm \infty}\frac{1}{n}\log |D(f^sg^t)^n_x v|$$
		exists for every $v\in V(x)\setminus\{0\}$.
		
		For every $\varepsilon>0$, let
		$$A_{\varepsilon}=\Big\{x\in \Gamma: \lambda(x,v,f^sg^t)-s\lambda(f)-t\lambda(g)>4\varepsilon \ \text{for some} \ v\in V(x)\Big\}.$$
		Assume that  $\mu(A_{\varepsilon})>0$.  For every $\ell\in \N$, let
		\begin{align*}
			{A}'_{\varepsilon,\ell}=\Big\{x\in A_{\varepsilon}: \exists \ u_x \in  V(x) \ \text{s.t.}\ |D(f^sg^t)^n_x u_x|\geq \ell^{-1} e^{n(\lambda(x,u_x,f^sg^t)-\varepsilon)}|u_x|, \, \forall \ n\geq 1\Big \}
		\end{align*}
		and
		\begin{align*}
			{A}''_{\varepsilon,\ell}=\Big\{x\in A_{\varepsilon}: \ |Dg^{tn}_x u|\leq \ell e^{n(t\lambda(g)+\varepsilon)}|u|,  \ \forall \ u \in  V(x) \ \forall\ n\geq 1 \Big \}.
		\end{align*}
		Since $\mu(A_{\varepsilon})>0$, one may choose $\ell$ large enough such that $\mu({A}'_{\varepsilon,\ell}\cap {A}''_{\varepsilon,\ell})>0$. For every $x\in {A}'_{\varepsilon,\ell}\cap {A}''_{\varepsilon,\ell}$. Applying Poincar\'{e}'s recurrence theorem with respect to $f^s$, there exists $n>2\frac{\log\ell}{\varepsilon}$ such that $f^{sn} x\in {A}'_{\varepsilon,\ell}\cap {A}''_{\varepsilon,\ell}$ and
		$$|Df^{sn}_x u|\leq |u|\cdot e^{n(s\lambda(f)+\varepsilon)}$$
        for every $u\in V(x)$.
		Hence,  for every $0\neq u\in V(x)$ we have that
		\begin{align*}
			|D(f^sg^t)^n_x u|&=|Dg^{tn}_{f^{sn}(x)}Df^{sn}_x u|\\
			&\leq \ell e^{n(t\lambda(g)+s\lambda(f)+2\varepsilon)}|u|\\
			&\leq \ell e^{-n\varepsilon} e^{n(\lambda(x,u,f^sg^t)-\varepsilon)}|u|\\
			&< \ell^{-1} e^{n(\lambda(x,u,f^sg^t)-\varepsilon)}|u|,
		\end{align*}
		which contradicts with the fact $x\in {A}'_{\varepsilon,\ell}$. Hence, we have $\mu(A_{\varepsilon})=0$ for every $\varepsilon> 0$. In a similar fashion, one can show that the following set
		$$B_{\varepsilon}=\Big\{x\in \Gamma: \lambda(x,v,f^sg^t)-s\lambda(f)-t\lambda(g)<4\varepsilon \ \text{for some} \ v\in V(x)\Big\}.$$
		satisfies that $\mu(B_{\varepsilon})=0$ for every $\varepsilon> 0$. By the arbitrariness of $\varepsilon> 0$,  we have  that
		\begin{equation}\label{eq:fgV}
			\lim_{n\rightarrow \pm \infty}\frac{1}{n}\log |D(f^sg^t)^n_x v|=s\lambda(f)+t\lambda(g)\quad  \mu-a.e.\, x
		\end{equation}
		for every $0\neq v\in V(x)$.
	\end{proof}
	
	Fix $i\in \{1,\cdots,r(f)\}$, note that
	$$Dg_x E_{i}(x,f)=E_{i}(g(x),f) \ \text{,} \ \log^{+} |Dg_x|_{E_{i}(x,f)}|\in L^1(\mu)\, \text{and}\, \dim E_{i}(x,f)=m_i(f) .$$
	Restrict on the finite dimensional $Dg$-invariant sub-bundle $E_{i}(\cdot,f)$, and apply the multiplicative ergodic theorem for the cocycle $\{Dg^{n}\}_{n\ge 0}$ which is defined on $A$ with respect to $g$,  for $\mu$-almost every $x\in \Gamma$ there exist $j_i$ numbers
	$$\lambda_{i,1}(g)>\lambda_{i,2}(g)>\cdots>\lambda_{i,j_i}(g)\geq -\infty$$
	and a unique splitting of the subspace $E_{i}(x,f)$
	\begin{equation}\label{eq:E31}
		E_{i}(x,f)=E_{i,1}(x) \bigoplus \cdots \bigoplus E_{i,j_i}(x)
	\end{equation}
	such that the following properties hold:
	\begin{enumerate}
		\item[(a)] for each $j=1,\cdots,j_i$, one has that $\dim E_{i,j}(x)=m^{E}_{i,j}\leq m_i(f)$, $Dg_x E_{i,j}(x)=E_{i,j}(g(x))$,  and $Df_x E_{i,j}(x)=E_{i,j}(f(x))$ by using a similar arguments as in Lemma \ref{L2}. Moreover, for every $v\in E_{i,j}(x)\setminus \{0\}$ we have that
		\begin{align*}
		\lim_{n\rightarrow \pm \infty} \frac{1}{n}\log |Df^{n}_x v|=\lambda_{i}(f)\,\text{and}\,
		\lim_{n\rightarrow \infty} \frac{1}{n}\log |Dg^{n}_x v|=\lambda_{i,j}(g)
		\end{align*}
		the second limit also holds for $n\rightarrow -\infty$ if $\lambda_{i,j}(g)>-\infty$.
		\item[(b)] the map $x\mapsto E_{i,j}(x)$ ($j=1,\cdots,j_i$) is $\mu$-continuous.
	\end{enumerate}
    \begin{remark}
    	(1)There is a natural splitting of $E_{i}(x,f)$
    	$$E_{i}(x,f)=E_{i,1}(x) \bigoplus \cdots \bigoplus E_{i,r(g)}(x) \bigoplus E_{i,\alpha}(x)$$
    	where $E_{i,j}(x)=E_{i}(x,f)\cap E_{j}(x,g)$ and
    	$E_{i,\alpha}(x)=E_{i}(x,f)\cap E_{\alpha}(x,g)$.
    	However, one can only show that the Lyapunov exponents of $D(fg)$ on the finite dimensional sub-bundle $E_{i,\alpha}$ are upper bounded by $\lambda_{\alpha}+\lambda_i(f)$, which may be larger than $0$. This implies that $E_{i,\alpha}$ may not belong to any subspace which is unstable, center or stable with respect to $D(fg)$. We would like to point out that the above method of constructing the desired splitting is slightly different with \cite{Li20}; (2) If $\log^{+} |(Dg_x|_{E_{i}(x,f)})^{-1}|\in L^1(\mu)$, then the Lyapunov exponent $\lambda_{i,j_i}(g)$ is finite.
    \end{remark}

	Note that the Lyapunov exponents of $Df,Dg$ on the sub-bundle $E_{i,j}$ are $\lambda_{i}(f)$ and $\lambda_{i,j}(g)$ respectively, it follows from Lemma \ref{L4} that the  Lyapunov exponent of $D(fg)$ on $E_{i,j}$ is $\lambda_{i}(f)+\lambda_{i,j}(g)$.
	
	For $x\in \Gamma$, let
	$$E_{\alpha,j}(x):=E_{\alpha}(x,f)\cap E_{j}(x,g) \ \text{for} \ j=1,\cdots,r(g),$$
	and
	$E_{\alpha}(x):=E_{\alpha}(x,f)\cap E_{\alpha}(x,g)$.
	Recall that
	$$\limsup_{n\rightarrow \infty}\frac{1}{n}\log|Df^n_x|_{E_{\alpha}(x)}|\leq \lambda_{\alpha}\ \text{and} \
	  \limsup_{n\rightarrow \infty}\frac{1}{n}\log|Dg^n_x|_{E_{\alpha}(x)}|\leq \lambda_{\alpha}.$$
    Therefore, using the same arguments as in the proof of Lemma \ref{L4}, one can show that
	\begin{equation}\label{eq:fgE-alpha}
	\limsup_{n\rightarrow \infty}\frac{1}{n}\log|D(f^sg^t)^n_x|_{E_{\alpha}(x)}|\leq (s+t)\lambda_{\alpha}\quad \forall s,t\in \N.
	\end{equation}
	
	Given the sub-bundle $E_{\alpha,j}$, the Lyapunov exponent of $Dg$ on it is $\lambda_j(g)$ for $j=1,\cdots,r(g)$. To get the Lyapunov exponents of $Df$ and $D(fg)$ on it, applying the multiplicative ergodic theorem for cocycle $\{Df^n|_{E_{\alpha,j}}\}_{n\ge 0}$ which is defined on $A$ with respect to $f$ and satisfied that $\log^{+}|Df |_{E_{\alpha,j}}|\in L^1(\mu)$,
	 for $\mu$-a.e. $x\in \Gamma$ there exist $i_j$ numbers
	$$(\lambda_{\alpha}>)\lambda_{1,j}(f)>\lambda_{2,j}(f)>\cdots>\lambda_{i_j,j}(f)\geq -\infty$$
	and a unique splitting of the subspace $E_{\alpha, i}(x,f)$
	\begin{equation}\label{eq:E32}
		E_{\alpha,i}(x,f)=G_{1,j}(x) \bigoplus \cdots \bigoplus G_{i_j,j}(x)
	\end{equation}
	such that  the following properties hold:
	\begin{enumerate}
		\item[(a)] for each $i=1,\cdots,i_j$, one has that $\dim G_{i,j}(x)=m^{G}_{i,j}\leq m_j(g)$, $Df_x G_{i,j}(x)=G_{i,j}(f(x))$, and $Dg_x G_{i,j}(x)=G_{i,j}(g(x))$  by using a similar arguments as in Lemma \ref{L2}. Moreover, for every  $v\in G_{i,j}(x) \setminus \{0\}$ we have that
		\begin{align*}
		\lim_{n\rightarrow \infty} \frac{1}{n}\log |Df^{n}_x v|&=\lambda_{i,j}(f)\\
		\lim_{n\rightarrow \pm \infty} \frac{1}{n}\log |Dg^{n}_x v|&=\lambda_{j}(g),
		\end{align*}
		where the first limit also holds for $n\rightarrow -\infty$ if $\lambda_{i,j}(f)>-\infty$;
		\item[(b)] the map $x\mapsto G_{i,j}(x)$ ($i=1,\cdots,i_j$) is $\mu$-continuous.
	\end{enumerate}

    Together with (\ref{eq:E31}) and (\ref{eq:E32}),  for $\mu$-almost every $x$ we get a both $Df$-invariant and $Dg$-invariant splitting of $\cB_x$:
    $$\cB_x=\bigr(\bigoplus_{i=1}^{r(f)} \bigoplus_{j=1}^{j_i} E_{i,j}(x)\bigr)
    \bigoplus\bigr(\bigoplus_{j=1}^{r(g)} \bigoplus_{i=1}^{i_j} G_{i,j}(x)\bigr)
    \bigoplus E_{\alpha}(x)$$ such that the following properties hold:
    \begin{enumerate}
    	\item[(i)] for each $i=1,\cdots,r(f)$, $j=1,\cdots,j_i$, the Lyapunov exponents of $Df$,$Dg$ and $D(fg)$ on the sub-bundle $E_{i,j}$ are $\lambda_{i}(f)$, $\lambda_{i,j}(g)$ and $\lambda_{i}(f)+\lambda_{i,j}(g)$ respectively
    	($\lambda_{i,j_i}(g)$ may be $-\infty$);
    	\item[(ii)] for each $j=1,\cdots,r(g)$, $i=1,\cdots,i_j$, the Lyapunov exponents of $Df$,$Dg$ and $D(fg)$ on the sub-bundle $G_{i,j}$ are $\lambda_{i,j}(f)$, $\lambda_{j}(g)$ and $\lambda_{i,j}(f)+\lambda_{j}(g)$ respectively
    	($\lambda_{i_j,j}(f)$ may be $-\infty$);
    	\item[(iii)] $\displaystyle{\limsup_{n\rightarrow \infty}\frac{1}{n}\log|D(f^sg^t)^n_x|_{E_{\alpha}(x)}|\leq (s+t)\lambda_{\alpha}}$ for every $s,t\in \N$.
    \end{enumerate}	

    \subsection{Norms of projections}
    In this section, we shall show the norms of projections are temperate.

	\begin{lemma}\label{L2.5}
	Let $\cB_x=V(x)\oplus U(x)$ be a measurable splitting of the tangent space $\cB_x$ for every $x\in \Gamma$, and let $\pi_{V}(x)$ be the projection of $\cB_x $ onto $ V(x) $ via the splitting.  Suppose that $V$ is a finite dimensional $Dh$-invariant sub-bundle and $Dh_x U(x)\subset U(h(x))$ for every $x\in \Gamma$, and $\log |(Dh|_{V})^{-1}|\in L^1(\mu)$.
	Then we have that
	$$\lim_{n\rightarrow \pm\infty} \frac{1}{n}\log |\pi_{V}(h^{n}(x))|=0 \quad \mu\text{-}a.e.\, x$$
    where $h=f$ or $g$.
    \end{lemma}
    \begin{proof}
	Let $h=f$ or $g$. Since $Dh_x V(x)=V(h(x))$ and $Dh_x$ is injective for every $x$, we have that
	$$Dh_x \circ (\pi_{V}(x)) v= (\pi_{V}(hx)) \circ Dh_x v$$
    for every $v\in \cB_x$.
	Hence, we have that $|\pi_{V}(x)|\leq |(Dh_x|_{V(x)})^{-1}| \cdot |\pi_{V}(h(x))| \cdot |Dh_x|$. Consequently, one has that
	$$\log |\pi_{V}(x)|-\log|\pi_{V}(h(x))|\leq \log |(Dh_x|_{V(x)})^{-1}|+ \log C_0$$
    where $C_0=\max\{\sup_{x\in A}|Df_x|,\sup_{x\in A}|Dg_x|\}<\infty$.
	Since $\log^{+}|(Dh_{x}|_{V(x)})^{-1}|\in L^1(\mu)$, we have that $(\log |\pi_{V}(x)|-\log|\pi_{V}(h(x))|)^{+}\in L^1(\mu)$.
    It follows from  Lemma \ref{Tem} that
	$$\lim_{n\rightarrow \pm\infty} \frac{1}{n}\log |\pi_{V}(h^{n}(x))|=0 \quad \mu\text{-}a.e.\,x.$$
    \end{proof}

    Let $\pi_{i,j}^{E}(x)$, $\pi_{i,j}^{G}(x)$ and $\pi_{\alpha}(x)$ denote the projections of $\cB_x$ onto
    $E_{i,j}(x)$, $G_{i,j}(x)$ and $E_{\alpha}(x)$ respectively, we have the following result.
    \begin{corollary}\label{L3}
	For every $i=1,\cdots,r(f)$, $j=1,\cdots,j_i$ with $\lambda_{i,j}(g)>-\infty$ and every $j=1,\cdots,r(g)$, $i=1,\cdots,i_j$ with $\lambda_{i,j}(f)>-\infty$, we have that
	$$
	\lim_{n\rightarrow \pm\infty} \frac{1}{n}\log |\pi_{i,j}^{E}(h^{n}(x))|=0 \ \text{and} \
	\lim_{n\rightarrow \pm\infty} \frac{1}{n}\log |\pi_{i,j}^{G}(h^{n}(x))|=0
	\quad \mu\text{-}a.e.\, x$$
    where $h=f$ or $g$.
    \end{corollary}

    \begin{proof}
	For every $i\in \{1,\cdots,r(f)\}$, $j\in \{1,\cdots,j_i\}$ and every $v\in E_{i,j}(x)\setminus\{0\}$, we have that
	$$\lim_{n\rightarrow \pm \infty} \frac{1}{n} \log |Df^{n}_{x} v|=\lambda_{i}(f), \ \lim_{n\rightarrow \pm \infty} \frac{1}{n} \log |Dg^{n}_{x} v|= \lambda_{i,j}(g) $$
	and $\lambda_{i}(f)>\lambda_{\alpha}$, $\lambda_{i,j}(g)>-\infty$ by the assumptions. By Lemma \ref{L1.5}, one has that
	$$\log|(Df|E_{i,j})^{-1}|\in L^1(\mu) \ \text{and} \
	\log|(Dg|E_{i,j})^{-1}|\in L^1(\mu).$$
	It follows form Lemma \ref{L2.5} that the norm of projection $\pi_{i,j}^{E}$ is ($\mu$-)temperate with respect to $f$ and $g$ respectively. In a similar fashion, one can show that the norm of projection $\pi_{i,j}^{G}$ is  ($\mu$-)temperate with respect to $f$ and $g$ respectively.
    \end{proof}

    Furthermore, if $\mu$ satisfies the condition (H4) then the norm of $\pi_{\alpha}(x)$ is also temperate.
    This condition is also used in the two cases:  (1) $i=1,\cdots,r(f)$ and  $j= j_i$; (2)$j=1,\cdots,r(g)$ and $i=i_j$.
    \begin{corollary}\label{C3}
	If $\mu$ satisfies the condition (H4), then the exponent $\lambda_{i,j}(f)$ is finite for every $i=1,\cdots,r(f)$ and $j=1,\cdots, j_i$, and the exponent
    $\lambda_{i,j}(g)$ is finite for every $j=1,\cdots,r(g)$ and $i=1,\cdots, i_j$. Moreover, all the norms of the projections $\pi_{i,j}^{E}(x)$, $\pi_{i,j}^{G}(x)$ and $\pi_{\alpha}(x)$ are ($\mu$-)temperate with respect to $f$ and $g$ respectively.
    \end{corollary}
    \begin{proof} 
    Note that the Lyapunov exponent of the cocycle $\{Df_x^n\}_{n\ge 0}$ on the sub-bundle $E_i(x,f)$ is $\lambda_i(f)$, and the Lyapunov exponent of the cocycle $\{Dg_x^n\}_{n\ge 0}$ on the sub-bundle $E_{\alpha,j}(x)$ is $\lambda_j(g)$, both $\lambda_i(f)$ and $\lambda_j(g)$ are finite.
	By Lemma \ref{L1.5}, we have that
	$$ \int \log^{+}|(Df_x|E_i(x,f))^{-1}| d\mu<\infty \ \text{and} \
	\int \log^{+}|(Dg_x|E_{\alpha,j}(x))^{-1}| d\mu<\infty$$
    for every $1\leq i \leq r(f),1\leq j \leq r(g)$.
 	Consequently,  by assumption (H4) we have that
	$$ \int \log^{+}|(Dg_x|E_i(x,f))^{-1}| d\mu<\infty \ \text{and} \
	\int \log^{+}|(Df_x|E_{\alpha,j}(x))^{-1}| d\mu<\infty$$
    for every $1\leq i \leq r(f),1\leq j \leq r(g)$. The integrability of  $\log^{+}|(Dg_x|E_i(x,f))^{-1}|$ implies that $\lambda_{i,j}(g)>-\infty$ for every $i=1,\cdots,r(f)$ and $j=1,\cdots, j_i$. Similarly, the  integrability of $ \log^{+}|(Df_x|E_{\alpha,j}(x))^{-1}|$ yields that $\lambda_{i,j}(f)>-\infty$ for every $j=1,\cdots,r(g)$ and $i=1,\cdots, i_j$.

   Hence, by Corollary \ref{L3} the norms of the projections $\pi_{i,j}^{E}(x)$ and $\pi_{i,j}^{G}(x)$ are ($\mu$-)temperate with respect to $f$ and $g$ respectively. Finally, the fact $$Id_x=\sum_{i=1}^{r(f)}\sum_{j=1}^{j_i}\pi_{i,j}^{E}(x)+\sum_{j=1}^{r(g)}\sum_{i=1}^{i_j}\pi_{i,j}^{G}(x)
   +\pi_{\alpha}(x)$$
   yields that the norm of the projection $\pi_{\alpha}(x)$ is ($\mu$-)temperate with respect to $f$ and $g$ respectively.
   \end{proof}
   \begin{remark}
	In general, there are two ways to prove the norm of the projection operator of some subspace is temperate. The first one is that we use to prove the above  lemma, which is suitable to the case that the Lyapunov exponents on it are bounded. We refer the reader to \cite[Lemma 1]{Froyland18}, \cite[Lemma 2.11]{Quas12} and \cite[Proposition 3.7]{Lucas17} for the second one, where the condition is that the exponential growth rate of  vectors in the subspace is 
    not equal to the exponential growth rate of vectors  in the complement of the subspace.
   \end{remark}

\subsection{MET for $Df$ and $Dg$}	
    Without loss of generality, assume that $\Gamma$ is both $f$-invariant and $g$-invariant, and all of the previous statements hold for every $x\in \Gamma$. Otherwise, we may remove a zero  measure subset of $\Gamma$ to achieve this goal. Then the MET for commuting transformations on Banach spaces can be stated  as follows.
	\begin{theorem}\label{MET fg}
		Given a separable Banach space $\cB$ with norm $|\cdot|$. Let $A,f,g$ satisfy conditions (H1)-(H3) and let $\mu \in \cE(f,g,A)$. Then, there exist an  $f$-invariant and $g$-invariant Borel subset $\Gamma$ with $\mu(\Gamma)=1$, and numbers
		\begin{align*}
		\lambda_1 (f)>\lambda_2(f)>\cdots \lambda_{r(f)} (f)>\lambda_{\alpha} \ \text{and} \
		\lambda_1 (g)>\lambda_2(g)>\cdots \lambda_{r(g)} (g)>\lambda_{\alpha},
		\end{align*}
		such that for every $1\leq i \leq r(f),1\leq j \leq r(f)$ there exist numbers
		\begin{align*}
		\lambda_{i,1} (g)>\lambda_{i,2}(g)>\cdots \lambda_{i,j_i} (g)\geq -\infty \ \text{and} \
		\lambda_{1,j} (f)>\lambda_{2,j}(f)>\cdots \lambda_{i_j,j} (f)\geq -\infty,
		\end{align*}
		so that for every $x\in \Gamma$ we have a splitting of $\cB_x$:
		$$\cB_x=(\bigoplus_{i=1}^{r(f)} \bigoplus_{j=1}^{j_i} E_{i,j}(x))
		\bigoplus(\bigoplus_{j=1}^{r(g)} \bigoplus_{i=1}^{i_j} G_{i,j}(x))
		\bigoplus E_{\alpha}(x)$$
		with the following properties:
		\begin{enumerate}
			\item[(a)] For every $1\leq i \leq r(f)$ and $1\leq j\leq j_i$, $\dim E_{i,j}(x)=m^E_{i,j}<\infty$ and for every $s,t\in \Z$, $D(f^sg^t)_x E_{i,j}(x)=E_{i,j}(f^sg^t(x))$. If $\lambda_{i,j}(g)>-\infty$, then for every $v\in E_{i,j}(x)\setminus \{0\}$ we have that
			$$
			\lim_{n\rightarrow \pm \infty} \frac{1}{n}\log |D(f^sg^t)^{n}_x v|=s\lambda_i(f)+t\lambda_{i,j}(g).
			$$
			\item[(b)]For every $1\leq j \leq r(g)$ and $1\leq i\leq i_j$, $\dim G_{i,j}(x)=m^G_{i,j}<\infty$ and for every $s,t\in \Z$, $D(f^sg^t)_x G_{i,j}(x)=G_{i,j}(f^sg^t(x))$. If $\lambda_{i,j}(f)>-\infty$, then for every $v\in G_{i,j}(x)\setminus \{0\}$ we have that
			$$
			\lim_{n\rightarrow \pm \infty} \frac{1}{n}\log |D(f^sg^t)^{n}_x v|=t\lambda_j(g)+s\lambda_{i,j}(f).
			$$
			\item[(c)] The subspace $E_{\alpha}(x)$ is closed and finite co-dimensional, and satisfies that
			$D(f^sg^t)_x E_{\alpha}(x)\subset E_{\alpha}(f^sg^t(x))$ for every $s,t\in \N$ and
			$$ \limsup_{n\rightarrow \infty} \frac{1}{n}\log |D(f^sg^t)^{n}_x|_{E_{\alpha}(x)}|\leq (s+t)\lambda_{\alpha}; $$
			\item[(d)] The maps $x\mapsto E_{i,j}(x),x\mapsto G_{i,j}(x)$ and $x\mapsto E_{\alpha}(x)$ are  $\mu$-continuous;
			\item[(e)] Additionally, if $\mu$ satisfies the condition (H4), then for every pair $(i,j)$ the Lyapunov exponents $\lambda_{i,j}(f),\lambda_{i,j}(g)$ are finte. Let $\pi^E_{i,j}(x),\pi^G_{i,j}(x)$ and $\pi_{\alpha}(x)$ denote the projection of $\cB_x $ onto $E_{i,j}(x)$, $G_{i,j}(x)$ and $E_{\alpha}(x)$ via the splitting respectively, then the norm of each of those projection operators is ($\mu$-)temperate.
		\end{enumerate}
	\end{theorem}

	\section{\textbf{Lyapunov norms}}\label{SEC4}
	Let $f,g,A$ and $\mu$ be as in Theorem \ref{A},  we will construct a new norm $|\cdot|_x$ in the tangent spaces $\cB_x$   in this section, which is called the Lyapunov norm.
	
    By Theorem \ref{MET fg}, there exists a both $f$-invariant and $g$-invariant Borel subset $\Gamma\subset A$ with $\mu(\Gamma)=1$ and a splitting of $\cB_x$:
	\begin{align}\label{eq:split}
	\cB_x=(\bigoplus_{i=1}^{r(f)} \bigoplus_{j=1}^{j_i} E_{i,j}(x))
	\bigoplus(\bigoplus_{j=1}^{r(g)} \bigoplus_{i=1}^{i_j} G_{i,j}(x))
	\bigoplus E_{\alpha}(x)
	\end{align}
    For simplifying the notations, we rewrite the splitting $\eqref{eq:split}$ as
    $$\cB_x=\bigoplus_{\kappa\in \cI} V_{\kappa}(x)\bigoplus E_{\alpha}(x) $$
	where $\cI$ is some finite index set with $\{V_{\kappa}\}_{\kappa \in \cI}=\{E_{i,j}\}_{(i,j)}\cup \{G_{i,j}\}_{(i,j)}$.
	
	For $h=f,g$ or $fg$ and every $\kappa\in \cI$, let $\lambda^{\kappa}(h)$ denote  the Lyapunov exponent of $Dh$ on the sub-bundle $V_{\kappa}$. Given $x\in \Gamma$, put
	\begin{align*}
	E^{u}(x,h)=\bigoplus\{V_{\kappa}(x):&\lambda^{\kappa}(h)>0\},\\
	E^{c}(x,h)=\bigoplus\{V_{\kappa}(x):&\lambda^{\kappa}(h)=0\},\\
	E^{s}(x,h)=(\bigoplus\{V_{\kappa}(x):& \lambda^{\kappa}(h)<0\})
	\bigoplus E_{\alpha}(x).
	\end{align*}
	Then,  we have that $\cB_x=E^{u}(x,h) \bigoplus E^{c}(x,h) \bigoplus E^{s}(x,h)$ and each subspace $E^{\tau}(x,h)$ $(\tau=u,c,s)$ is the direct sum of some subspaces in the splitting (\ref{eq:split}).
	Let
	\begin{align*}
	\lambda^{+}(h)=
	\min_{\kappa\in \cI} \{\lambda^{\kappa}(h): \lambda^{\kappa}(h)>0\} \end{align*}
    and
    \begin{align*}
	\lambda^{-}(h)=
	\max \bigr\{\max_{\kappa\in \cI} \{\lambda^{\kappa}(h):  \lambda^{\kappa}(h)<0\}, \lambda_{\alpha}(h)\bigr\}
	\end{align*}
	where $\lambda_{\alpha}(f)=\lambda_{\alpha}(g)=(\lambda_{\alpha}(fg))/2=\lambda_{\alpha}.$
	
	Let $K_0$ be the cardinality of the  set $\cI\bigcup\{\alpha\}$, fix a small number $\varepsilon$ such that
	\begin{equation}\label{eq:vp}
	0<\varepsilon<\frac{1}{10K_0}\min\bigr\{-\lambda^{-}(h),\lambda^{+}(h):h\in\{f,g,f g\}\bigr\}.
	\end{equation}
    For $x\in \Gamma$,  let
    \[|v|_{x}=\sum_{n,k\in \Z}\dfrac{|D(f^ng^k)_x v|}{e^{n\lambda^{\kappa}(f)+k\lambda^{\kappa}(g)+(|n|+|k|)\varepsilon}}\]
    if $v\in V_{\kappa}(x)$ for some $\kappa\in\cI$, and let
    \[|v|_{x}=\sum_{n,k\in \N}\dfrac{|D(f^ng^k)_x v|}{e^{(n+k)(\lambda_{\alpha}+\varepsilon)}}\]
    if $v\in E_{\alpha}(x)$.
	
	\begin{definition}
    Fix $\varepsilon>0$  as  above, define a set of point dependent norms $\{|\cdot|_{x}\}_{x\in \Gamma}$ as follows:
		\begin{equation*}
		|v|_{x}:=\max\{\max_{\kappa\in\cI}\{|v_{\kappa}|_x\},|v_{\alpha}|_x\}
				\end{equation*}
    where $v=\sum_{\kappa\in\cI}v_{\kappa}+v_{\alpha}\in \cB_x$ with $v_{\kappa}\in V_{\kappa}(x)$ for every $\kappa\in\cI$ and $v_{\alpha}\in E_{\alpha}(x)$.
	\end{definition}

    To see the definition of the above norm $|\cdot|_{x}$ is well-defined, it suffices to show that $|v|_{x}<\infty$ for every $v\in V_{i}(x)$ and every $i\in \cI$ and $i=\alpha$. Fix $\kappa\in\cI$ and $v\in V_{\kappa}(x)$, we will show that $|v|_x<\infty$. The case of $i=\alpha$ can be shown in a similar fashion. We first assume that $|n|\geq|k|\geq0$ and $n\cdot k\geq 0$. Let $\ell\in \N$ be large enough so that
	$$\sup_{x\in A}|Df_x|\leq e^{\lambda^{\kappa}(f)	  +\ell\cdot\frac{\varepsilon}{4}} \, \text{and}\,
	\sup_{y\in A}|Df_y|\leq e^{\lambda^{\kappa}(g)	+\ell\cdot\frac{\varepsilon}{4}}.$$
	Note that for every $0\leq t \leq \ell$, the following limit
	$$\lim_{s\rightarrow \pm \infty}\frac{1}{s} \log|D(f^\ell g^t)^s_xu|=\ell\lambda^{\kappa}(f)	+t\lambda^{\kappa}(g)	 $$
    is uniformly for $u\in V_{\kappa}(x)$ with $|u|=1$.
	Therefore, there exists $s_0>\ell$  such that for every $|s|>s_0$, one has that
	$$ |D(f^\ell g^t)^s_xu| \leq |u| \cdot
	e^{s(\ell\lambda^{\kappa}(f)	+t\lambda^{\kappa}(g))	+|s|\frac{\varepsilon}{4}}$$
	 for every $0\leq t\leq \ell$ and every $u\in V_{\kappa}(x)$.
	Let $N_\ell=s_0 \cdot \ell$. For every $|n|>N_\ell$, write $n=s\ell+p$ with $s\in \Z$ and $0\leq p \leq \ell$, and write $k=st+q$ with $0\leq q \leq |s|$. Since $|n|\geq|k|\geq0$ and $n\cdot k\geq 0$, one has $0\leq t \leq \ell$.
	Therefore, for every $|n|>N_\ell$ we have that
	\begin{equation*}
		\begin{aligned}
			|D(f^ng^k)_x u|&=|Df^p\circ Dg^q\circ  D(f^\ell g^t)^s_xu| \\
			&\leq |u|e^{n\lambda^{\kappa}(f)	+k\lambda^{\kappa}(g)	+(p\ell+q\ell+|s|)\cdot\frac{\varepsilon}{4} }\\
			&\leq |u|e^{n\lambda^{\kappa}(f)	+k\lambda^{\kappa}(g)	+(|n|+|k|)\cdot\frac{\varepsilon}{2}}.
		\end{aligned}
	\end{equation*}
    For $|n|\geq|k|\geq0$ and $nk<0$, the above inequality  can also be proven for $|n|>N_\ell$ by some modification of the choice of $t,s\in \Z$ and $p,q,N_\ell>0$. If $|k|\geq|n|$, one can prove in a similar fashion that  there exists a sufficiently large $K_\ell$ so that the above inequality holds for $|k|>K_\ell$.
	Therefore, we have that
	$$|D(f^ng^k)_x u|\leq |u|e^{n\lambda^{\kappa}(f)	+k\lambda^{\kappa}(g)	 +(|n|+|k|)\cdot\frac{\varepsilon}{2}}$$
	for every $(n,k)\in \mathbb{Z}^2$ with $|n|>N_\ell$ and $|k|>K_\ell$.
	Hence, we have
	$$\sum_{n,k\in \Z}\dfrac{|D(f^ng^k)_xu|/|u|}{e^{n\lambda^{\kappa}(f)	+k\lambda^{\kappa}(g)	 +(|n|+|k|)\varepsilon}}<\infty.$$
	This yields that the norm $|u|_x$ is well-defined.

    To estimate the difference of these new norms and the original ones, put
    \begin{align*}
	C_{\kappa}(x)=\sup_{n,k\in \Z}\dfrac{\sup_{v\in V_{\kappa}(x),|v|=1}|D(f^ng^k)_x v|}{e^{n\lambda^{\kappa}(f)+k\lambda^{\kappa}(g)+(|n|+|k|)\cdot\frac{\varepsilon}{2}}}
    \end{align*}
    if $\kappa\in \cI$, and
    \begin{align*}
	C_{\alpha}(x)=\sup_{n,k\geq 0}\dfrac{\sup_{v\in E_{\alpha}(x),|v|=1}|D(f^ng^k)_x v|}{e^{(n+k)(\lambda_{\alpha}+\frac{\varepsilon}{2})}}.\end{align*}
    Let $C(x)=\max\{\max_{\kappa\in \cI}\{C_{\kappa}(x),\pi_{\kappa}(x)\},C_{\alpha}(x),\pi_{\alpha}(x)\}$, where $\pi_{\kappa}(x)$ is the projection of $\cB_x $ onto $V_{\kappa}(x)$ via the splitting.
    Note that all of these functions on $\Gamma$ are Borel measurable and finite-valued, one can show the function $C(x)$ satisfies the following property by using (e) of Theorem \ref{MET fg} and Lemma \ref{Tem}. 
    \begin{lemma} \label{L41} Let the function $C(x)$ be given as above, then
		\begin{equation}\label{eq:Cx}
		\lim_{n\rightarrow \pm \infty} \frac{1}{n} \log C(h^n(x))=0
	\end{equation}
    for $\mu$-almost every $x\in \Gamma$, where $h=f,g$ or $fg$.
    \end{lemma}

    Using the similar arguments in \cite[Lemma 5.1]{Young17}, we have the following result.
     \begin{lemma}\label{L42}
	  The following properties hold for all $x\in \Gamma$:
	 \begin{enumerate}
		\item[(1)] For every $v\in V_{\kappa}(x),\kappa\in \cI$, and every $w\in E_{\alpha}(x)$,
                   we have that
		\begin{align*}
			e^{-(|n|+|k|)\frac{\varepsilon}{2}}|v|_{x}
			\leq \frac{|D(f^ng^k)_x v|_{f^ng^k(x)}}{ e^{n\lambda^{\kappa}(f)+k\lambda^{\kappa}(g)}}
			\leq e^{(|n|+|k|)\frac{\varepsilon}{2}}|v|_{x} \quad \forall \  n,k\in\Z
     \end{align*}  and
     \begin{align*}
			|D(f^ng^k)_x w|_{f^ng^k(x)} &\leq e^{(n+k)(\lambda_{\alpha}+\frac{\varepsilon}{2})}|w|_{x} \quad \forall \  n,k\geq 0;
     \end{align*}
		\item[(2)] The  norms $|\cdot|_{x}$ and $|\cdot|$ satisfy that
		$$\frac{1}{K_0}|\cdot|\leq |\cdot|_{x} \leq \bigr(\frac{6C(x)}{1-e^{-\frac{\varepsilon}{2}}}\bigr)^2 |\cdot|$$
	\end{enumerate}
                   where $K_0$ is the cardinality of the index set $\cI\bigcup\{\alpha\}$.
    \end{lemma}

	Consequently, we have the following  result by considering $n=1,k=0$ or $n=0,k=1$ or $n=1,k=1$.
	\begin{corollary}\label{C41}
		For $h=f,g$ or $fg$, the following properties hold for all $x\in \Gamma$:
		\begin{align}
		e^{\lambda^{+}(h)-\varepsilon}|u|_{x} \leq &|Dh_x u|_{h(x)} \label{eq:41};  \\
		e^{-\varepsilon}|v|_{x} \leq &|Dh_x v|_{h(x)} \leq e^{\varepsilon}|v|_{x} \label{eq:42};\\
		&|Dh_x w|_{h(x)} \leq e^{\lambda^{-}(h)+\varepsilon}|w|_{x} \label{eq:43}.
		\end{align}
		where $u\in E^{u}(x,h)$, $v\in E^{c}(x,h)$ and $w\in E^{s}(x,h)$.
	\end{corollary}
	
	Denote by $\exp_x:\cB_x\rightarrow \cB$ the exponential map defined by $v\mapsto v+x$. Define the connecting maps
	$\widetilde{h}_x:\cB_x \rightarrow \cB_{h(x)}$ by $ \widetilde{h}_x:=\exp^{-1}_{h(x)}\circ h \circ \exp_x,$
	and let $\widetilde{h}^{n}_x=\widetilde{h}_{h^{n-1}x}\circ \cdots \circ  \widetilde{h}_x$, where $h=f,g$ or $fg$. Note that  $(D\widetilde{h}_x)_{0}=Dh_x$.
    Since $f,g$ are $C^2$ transformations and $A$ is compact, there exist $M_0>0$ and $\gamma_0 >0$ such that
	\begin{equation}
	\max\{|D^2f_x|,|D^2g_x|,|D^2(fg)_x|\}<M_0
	\end{equation}
	for every $x\in \cB$ with $\dist(x,A)<\gamma_0$.

    For $h=f$ or $g$, by Lemma \ref{L41} and \cite[Lemma 3.1]{Hu96}, for each $\delta>0$, there exists a measurable function $K_{\delta}:\Gamma\rightarrow [1,\infty)$ so that
	\begin{equation}
	K_{\delta}(h^{\pm}(x))\leq e^{\delta}K_{\delta}(x), \ \text{and} \ K_{\delta}(x)\geq (\frac{6C(x)}{1-e^{-\frac{\varepsilon}{2}}})^2.
	\end{equation}
	Let $K(x):=K_{\frac{\varepsilon}{2}}(x)$, by Lemma \ref{L42} we have
	\begin{equation}\label{eq:norms}
	\frac{1}{K_0}|\cdot|\leq |\cdot|_{x} \leq K(x) |\cdot|.
	\end{equation}
    Let $\widetilde{B}_x(r):=\{v\in \cB_x:|v|_x\leq r\}$, we have the following result.
	
	\begin{lemma}\label{L43}
		Let the function $\ell:\Gamma\rightarrow [1,\infty)$ given by
		$\ell(x)=M_0K^2_0e^{\varepsilon}K(x)$ and $h=f,g$ or $fg$.
		For every $0<\delta<\gamma_0$  and every $x\in \Gamma$,   the map $\widetilde{h}_x: \widetilde{B}_x(\delta \ell(x)^{-1}) \rightarrow \widetilde{B}_{h(x)}$ satisfies the following properties:
		\begin{enumerate}
			\item[(1)] $\Lip(\widetilde{h}_x-(D\widetilde{h}_x)_{0})\leq \delta$;
			\item[(2)] the map $z\mapsto (D\widetilde{h}_x)_{z}$ satisfies $\Lip (D\widetilde{h}_x)\leq \ell(x).$
		\end{enumerate}
	\end{lemma}
	\begin{proof}
		For every $y,z\in \widetilde{B}_x(\delta \ell(x)^{-1})$, since
		\begin{align*}
		|(D\widetilde{h}_x)_{z}-(D\widetilde{h}_x)_{y}|&= \sup_{v\neq 0} \dfrac{|(D\widetilde{h}_x)_{z}v-(D\widetilde{h}_x)_{y}v|_{h(x)}}{|v|_{x}}\\
		&\leq
		K(h(x))K_0 |Dh_{\exp_{x} y}-Dh_{\exp_{x} z}|\\
		&\leq e^{\varepsilon}K_0M_0 K(x)|z-y|\\
		&\leq e^{\varepsilon}K^2_0M_0 K(x)|z-y|_{x}\\
        &\leq \ell(x)|z-y|_{x},
		\end{align*}
		this proves the second statement. To prove the first statement, note that
		\begin{align*}
		\Lip(\widetilde{h}_x-(D\widetilde{h}_x)_{0})
		&\leq \sup_{y\in \widetilde{B}_x(\delta \ell(x)^{-1})}|(D(\widetilde{h}_x-(D\widetilde{h}_x)_{0})_{y}|\\
		&= \sup_{y\in \widetilde{B}_x(\delta \ell(x)^{-1})}|(D\widetilde{h}_x)_{y}-(D\widetilde{h}_x)_{0}|\\
		&\leq  \Lip(D\widetilde{h}_x)\cdot|y-0|_{x} \\
        &\leq \ell(x)\cdot\delta \ell(x)^{-1}\\
        &= \delta,
		\end{align*}
	where the second statement is used in the fourth inequality. This completes the proof.
	\end{proof}

	\begin{remark}
		By the definition of $\ell(x)$, for $h=f,g$ or $fg$ we have
		$$\ell(x)\geq K(x) \ \text{and} \ \ell(h^{\pm}(x))\leq e^{\varepsilon}\ell(x)$$
		where $\varepsilon$ is fixed as in (\ref{eq:vp}). It follows from (\ref{eq:norms}) that
		\begin{equation}\label{eq:norm}
		\frac{1}{K_0}|\cdot|\leq |\cdot|_{x} \leq \ell(x) |\cdot|.
		\end{equation}
	\end{remark}
	
	Choose $\delta_1<\gamma_0$ small enough such that for every $0<\delta\leq \delta_1$, one has
	\begin{equation}\label{eq:delta}
	e^{\varepsilon}+\delta<e^{2\varepsilon}<e^{4\varepsilon}-\delta \ \text{and} \  e^{-\varepsilon}-\delta>e^{-2\varepsilon}>e^{-4\varepsilon}+\delta.
	\end{equation}
    Fix $\lambda>\log(e^{2\max\{\lambda_1(f),\lambda_1(g)\}+\varepsilon}+\delta_1)$.  For $h=f,g$ or $fg$, one can show that
       \begin{equation}\label{eq:Expanding} \begin{aligned}
    |\widetilde{h}_x u|_{h(x)} &\leq |\widetilde{h}_x u-(D\widetilde{h}_x)_{0}u|_{h(x)}+|(D\widetilde{h}_x)_{0}u|_{h(x)} \\
    &\leq (\delta+e^{\lambda_{\max}(h)+\varepsilon})  |u|_{x} \\
    &\leq e^{\lambda} |u|_{x}
    \end{aligned}
    \end{equation}
    for every $u\in \widetilde{B}_x(\delta_1\ell(x)^{-1})$.
	In particular, one has that \begin{eqnarray}\label{exp-1step}
    \widetilde{h}_x(\widetilde{B}_{x}(e^{-\lambda-\varepsilon}\delta_1 \ell(x)^{-1}))\subset \widetilde{B}_{h(x)}(\delta_1 \ell(h(x))^{-1})
    \end{eqnarray}for every $x\in \Gamma$.
	\begin{corollary}\label{C42}
		For $h=f,g$ or $fg$ and $x\in \Gamma$. Let  $\pi_{x,h}^{\tau}$  denote the projection operator of $\cB_x$ onto $E^{\tau}(x,h)$ via the splitting $\cB_x=E^{u}(x,h) \bigoplus E^{c}(x,h) \bigoplus E^{s}(x,h)$ where $\tau=u,c,s$. For every
		$u,v\in \widetilde{B}_x(\delta_1 \ell(x)^{-1})$ with $|u-v|_{x}=|\pi_{x,h}^{u}(u-v)|_{x}$, then we have that
		$$|\widetilde{h}_x u-\widetilde{h}_x v|_{h(x)}=|\pi_{h(x),h}^{u}(\widetilde{h}_x u-\widetilde{h}_x v)|_{h(x)}.$$
	\end{corollary}
	\begin{proof}
		By definition of the norm $|\cdot|_{x}$, it  suffices to show that
		$$|\pi_{h(x),h}^{u}(\widetilde{h}_x u-\widetilde{h}_x v)|_{h(x)}\geq  |\pi_{h(x),h}^{cs}(\widetilde{h}_x u-\widetilde{h}_x v)|_{h(x)}$$
		where $\pi_{x,h}^{cs}=\pi_{x,h}^{c}+\pi_{x,h}^{s}.$	
		Since $\lambda^{+}(h)\geq 5\varepsilon $ by the choice of $\varepsilon>0$ (see (\ref{eq:vp})), we have that
		\begin{align*}
		&|\pi_{h(x),h}^{u}(\widetilde{h}_x u-\widetilde{h}_x v)|_{h(x)}\\
		 \geq &|\pi_{h(x),h}^{u}((D\widetilde{h}_x)_{0}(u-v))|_{h(x)}-|\pi_{h(x),h}^{u}((\widetilde{h}_x-(D\widetilde{h}_x)_{0})u-(\widetilde{h}_x-(D\widetilde{h}_x)_{0})v)|_{h(x)}\\
		 \geq&|(D\widetilde{h}_x)_{0}(\pi_{x,h}^{u}(u-v))|_{h(x)}-|(\widetilde{h}_x-(D\widetilde{h}_x)_{0})u-(\widetilde{h}_x-(D\widetilde{h}_x)_{0})v|_{h(x)}\\
		\geq&e^{\lambda^{+}(h)-\varepsilon}|\pi_{x,h}^{u}(u-v)|_{x}-\delta_1|u-v|_{x} \\
		\geq&(e^{4\varepsilon}-\delta_1)|u-v|_{x}.
		\end{align*}
		where  (\ref{eq:41}) and Lemma \ref{L43} is used in the third inequality.
	 Similarly,  we have that
		\begin{align*}
		&|\pi_{h(x),h}^{cs}(\widetilde{h}_x u-\widetilde{h}_x v)|_{h(x)}\\
		 \leq&|\pi_{h(x),h}^{cs}((D\widetilde{h}_x)_{0}(u-v))|_{h(x)}+|\pi_{h(x),h}^{cs}((\widetilde{h}_x-(D\widetilde{h}_x)_{0})u-(\widetilde{h}_x-(D\widetilde{h}_x)_{0})v)|_{h(x)}\\
		 \leq&|(D\widetilde{h}_x)_{0}(\pi_{x,h}^{cs}(u-v))|_{h(x)}+|(\widetilde{h}_x-(D\widetilde{h}_x)_{0})u-(\widetilde{h}_x-(D\widetilde{h}_x)_{0})v|_{h(x)}\\
		\leq&e^{\varepsilon}|\pi_{x,h}^{cs}(u-v)|_{x}+\delta_1|u-v|_{x} \\
		\leq&(e^{\varepsilon}+\delta_1)|u-v|_{x}.
		\end{align*}
		The choice of $\delta_1$ (see (\ref{eq:delta})) implies that $e^{\varepsilon}+\delta_1< e^{2\varepsilon}< e^{4\varepsilon}-\delta_1$. This yields that
		$$|\widetilde{h}_x u-\widetilde{h}_x v|_{h(x)}=|\pi_{h(x),h}^{u}(\widetilde{h}_x u-\widetilde{h}_x v)|_{h(x)}.$$
	\end{proof}

	\begin{corollary}\label{C43}
		For $h=f,g$ or $fg$. Let $x\in \Gamma$, $n\in\Z$ and $u,v\in\cB_x$ with $\widetilde{h}^{k}_x(u),\widetilde{h}^{k}_x(v)\in \widetilde{B}_{h^{k}(x)}(\delta_1 \ell(x)^{-1})$ for $k=0,1,\cdots,n$,
		\begin{enumerate}
			\item[(1)]if $|u-v|_{x}=|\pi_{x,h}^{u}(u-v)|_x$, then
			\begin{align*} |\widetilde{h}^{k}_xu-\widetilde{h}^{k}_xv|_{h^{k}(x)}&=|\pi_{h^{k}(x),h}^{u}(\widetilde{h}^{k}_xu-\widetilde{h}^{k}_xv)|_{h^{k}(x)}\\
            &\geq (e^{4\varepsilon}-\delta_1)^{k}|u-v|_{x}\\
			&\geq e^{2k\varepsilon}|u-v|_{x}
			\end{align*}
        for $k=0,1,\cdots,n$;
			\item[(2)]if $|\widetilde{h}^{n}_xu-\widetilde{h}^{n}_xv|_{h^{n}(x)}=|\pi_{h^{n}(x),h}^{cs}(\widetilde{h}^{n}_xu-\widetilde{h}^{n}_xv)|_x$, then $$|\widetilde{h}^{k}_xu-\widetilde{h}^{k}_xv|_{h^{k}(x)}=|\pi_{h^{k}(x),h}^{cs}(\widetilde{h}^{k}_xu-\widetilde{h}^{k}_xv)|_{h^{k}(x)}\leq e^{2k\varepsilon}|u-v|_{x}$$	
        for $k=0,1,\cdots,n$.		
		\end{enumerate}
	\end{corollary}	
	\begin{proof}
		It follows from Corollary \ref{C41}, Lemma \ref{L43}, Corollary \ref{C42} and the choice of $\varepsilon, \ \delta_1$ (see (\ref{eq:vp}) and (\ref{eq:delta})).
	\end{proof}
	
	\begin{remark}\label{R1}
		Corollary \ref{C42} remains true if we replace $\pi_{\cdot,h}^{u}, \pi_{\cdot,h}^{cs}$  by $\pi_{\cdot,h}^{uc}, \pi_{\cdot,h}^{s}$ respectively, where $\pi_{\cdot,h}^{uc}=\pi_{\cdot,h}^{u}+\pi_{\cdot,h}^{c}$. In this case, the inequalities (1) and (2) in Corollary \ref{C43} become
		\begin{align*}
		 |\widetilde{h}^{k}_xu-\widetilde{h}^{k}_xv|_{h^{k}(x)}=|\pi_{h^{k}(x),h}^{uc}(\widetilde{h}^{k}_xu-\widetilde{h}^{k}_xv)|_{h^{k}(x)}\geq e^{-2k\varepsilon}|u-v|_{x};
		\end{align*}
		and
		\begin{align*}
		 |\widetilde{h}^{k}_xu-\widetilde{h}^{k}_xv|_{h^{k}(x)}&=|\pi_{h^{k}(x),h}^{s}(\widetilde{h}^{k}_xu-\widetilde{h}^{k}_xv)|_{h^{k}(x)}\leq (e^{-4\varepsilon}+\delta_1)^k |u-v|_{x}\\
		&\leq e^{-2k\varepsilon} |u-v|_{x}
		\end{align*}
		respectively.
	\end{remark}
	
	\section{\textbf{Sub-additivity of measure-theoretic entropies}}\label{S5}
	In this section, we will give the detailed proof of  Theorem \ref{A}. Let $f,g,A$ and $\mu$ be  as in Theorem \ref{A}.
	
	\subsection{Local entropy}\label{Sub51}
	Suppose $T$ and $S$ are commuting continuous maps on a compact metric space $X$. Let $ B(x,\delta) $ be a closed ball in $X$ centered at $x$ and of radius $\delta$. We call the set
	$$B_n(x,\delta,T)=\bigcap_{k=0}^{n}T^{-k}B(T^kx,\delta)$$
	to be an $(n,\delta,T) $-ball. Let $\mu$ be a $T$-invariant measure on $X$, by Brin-Katok's entropy formula \cite{Katok83}, the following limit
	\begin{align*}
	h_{\mu}(x,T):&=\lim_{\delta\rightarrow 0}\limsup_{n\rightarrow \infty} -\frac{1}{n} \log \mu(B_n(x,\delta,T))\\
	&=\lim_{\delta\rightarrow 0}\liminf_{n\rightarrow \infty} -\frac{1}{n} \log \mu(B_n(x,\delta,T))
	\end{align*}
	is well-defined for $\mu$-almost every $x\in X$, and the quantity $h_{\mu}(x,T)$ is called the local entropy of $T$ at point $x$. Moreover,  the function $h_{\mu}(x,T)$ is $T$-invariant, integrable and satisfies that $\int h_{\mu}(x,T) d\mu(x)=h_{\mu}(T)$.

	\begin{lemma}[\cite{Hu96} Lemma 6.1]\label{L51}
		Let $T$ and $S$ be two commuting continuous homeomorphisms on a compact metric space $X$ and $\mu\in \cM(T,S,X)$. Then, the function $h_{\mu}(x,T)$ is both $T$-invariant and $S$-invariant. Consequently, if $\mu$ is $(T,S)$-ergodic, then $h_{\mu}(x,T)=h_{\mu}(T)$ for $\mu$-almost every $x\in X$.
	\end{lemma}
	
	Let $\rho:X \rightarrow [0,\infty)$ be a measurable function. Define an $(n,\rho,T)$-ball at $x$ by
	$$B_n(x,\rho,T)=\bigcap_{k=0}^{n}T^{-k}B(T^kx,\rho(T^kx)).$$
	The following result is an extension of the Brin-Katok's entropy formula.
	\begin{proposition}[\cite{Hu96} Proposition 6.4]\label{P5}
		Let $\{\rho_{\delta}:\delta>0\}$ be a family of measurable functions on $X$ satisfying that
		\begin{enumerate}
			\item[(1)] $0<\rho_{\delta}(x)\leq \delta$ for every $x$ and $\rho_{\delta}$ monotonically decreases as $\delta\rightarrow 0;$
			\item[(2)]$\int \log \rho_{\delta} d\mu<\infty$ for every $\delta>0$.
		\end{enumerate}
		Then
		$$h_{\mu}(x,T)=\lim_{\delta\rightarrow 0}\limsup_{n\rightarrow \infty} -\frac{1}{n} \log \mu(B_n(x,\rho_{\delta},T))\quad \mu-\text{a.e.}\, x.$$
	\end{proposition}
	
	For $r\in(0,1)$ and $\delta>0$, denote by $N_n(\delta,r,T)$ the minimal number of $(n,\delta,T) $-balls covering a set of $\mu$-measure more than $1-r$. Assume that $\mu$ is $T$-ergodic, by the Katok's entropy formula  (see \cite[Theorem 1.1]{Katok80}), we have that
	$$h_{\mu}(T)=\lim_{\delta\rightarrow 0}\limsup_{n\rightarrow \infty} \frac{1}{n} \log N_n(\delta,r,T).$$
	In fact, the assumption $\mu$ is $T$-ergodic is not necessary. By Lemma \ref{L51}, the following lemma can be proven easily.
	\begin{lemma}[\cite{Hu96} Proposition 6.3]\label{L52}
		If $T$ and $S$ are commuting continuous maps on a compact metric space $X$, $\mu$ is $(T,S)$-ergodic. Then, for every $r\in(0,1)$ we have that
		$$h_{\mu}(T)=\lim_{\delta\rightarrow 0}\limsup_{n\rightarrow \infty} \frac{1}{n} \log N_n(\delta,r,T).$$
	\end{lemma}
	\begin{remark}
		Since $\mu\in \cE(T,S,X)$ is equivalent to $\mu\in \cE(TS,T,X)$ (see Proposition \ref{P2} (2)),
	for every $r\in(0,1)$ we have that 	$$h_{\mu}(TS)=\lim_{\delta\rightarrow 0}\limsup_{n\rightarrow \infty} \frac{1}{n} \log N_n(\delta,r,TS).$$
	\end{remark}
	
	\subsection{Relations between Bowen balls}\label{Sub52}
	Let  $\widetilde{h}_{x}, \widetilde{B}_{x}(r),\delta_1,\varepsilon$ and $\ell(x)$ be the same as in  Section \ref{SEC4}.
	
    In the following, we will prove some useful lemmas, they has the counterpart in  finite dimensional systems which are proved in \cite[Sect. 7]{Hu96}. For the case of finite dimensional systems, Hu \cite{Hu96} considered two commuting $C^{2}$ diffeomorphisms  $f,g$  on a compact Riemannian manifold,  so that a version of Lemma \ref{L43} and all the corollaries are also valid for $\widetilde{f}^{-1}, \widetilde{g}^{-1}$. In our case, although $f^{-1}$ is well-defined and continuous on the set $A$, it may not be defined outside $A$ and may not be differentiable even in $A$. So we need some new methods to prove the  following lemmas, which is slightly different from those in  \cite{Hu96}.

	\begin{lemma}\label{L511}
		Let  $\delta>0$ be a small number and $x\in \Gamma$. For $h=f,g$ or $fg$, let $u\in \cB_x$ with $\widetilde{h}^{k}_x(u)\in \widetilde{B}_{h^{k}_x(u)}(\delta_1 \ell(x)^{-1})$ for every $k=0,1,\cdots,n$,
		\begin{enumerate}
			\item[(1)] if $u\in \widetilde{B}_x(\delta e^{-2n\varepsilon})$ and $\widetilde{h}^{n}_x(u)\in  \widetilde{B}_{h^n(x)}(\delta)$, then
			$$\widetilde{h}^{k}_x(u)\in  \widetilde{B}_{h^k(x)}(\delta e^{-2(n-k)\varepsilon}) \qquad \forall \ k=0,1,\cdots,n;$$
			\item[(2)] if $u\in \widetilde{B}_x(\delta)$ and $\widetilde{h}^{n}_x(u)\in  \widetilde{B}_{h^n(x)}(\delta e^{-2n\varepsilon})$, then
			$$\widetilde{h}^{k}_x(u)\in  \widetilde{B}_{h^k(x)}(\delta e^{-2k\varepsilon}) \qquad \forall \ k=0,1,\cdots,n.$$
		\end{enumerate}	
	\end{lemma}
	\begin{proof}
		For simplicity, we
		write $\pi^{\tau}_{h^{k}(x),h}$, $|\cdot|_{h^{k}(x)}$ as
		$\pi^{\tau}_{k}$,  $|\cdot|_{k}$ respectively.
		
		Note that
		$$|\widetilde{h}^{k}_x(u)|_{k}=\max\{|\pi_{k}^{u}\widetilde{h}^{k}_x(u)|_{k}, |\pi_{k}^{cs}\widetilde{h}^{k}_x(u)|_{k}\}.$$
		If $ |\widetilde{h}^{k}_x(u)|_{k}=|\pi_{k}^{u}\widetilde{h}^{k}_x(u)|_{k}$. By (1) of Corollary \ref{C43}, we have that
		$$ \delta\geq|\widetilde{h}^{n}_x(u)|_{n}\geq e^{2(n-k)\varepsilon} |\widetilde{h}^{k}_x(u)|_{k}$$
		Thus, $|\widetilde{h}^{k}_x(u)|_{k}\leq \delta e^{-2(n-k)\varepsilon}$. If
		$|\widetilde{h}^{k}_x(u)|_{k}=|\pi_{k}^{cs}\widetilde{h}^{k}_x(u)|_{k}$. By (2) of Corollary \ref{C43}, we have that
		$$|\widetilde{h}^{k}_x(u)|_{h^k(x)}\leq e^{2k\varepsilon}|u|_{x} \leq \delta e^{-2(n-k)\varepsilon}.$$
		This proves the first statement.
		
		To prove the second statement, note that $$|\widetilde{h}^{k}_x(u)|_{h^k(x)}=\max\{|\pi_{k}^{uc}\widetilde{h}^{k}_x(u)|_{k},|\pi_{k}^{s}\widetilde{h}^{k}_x(u)|_{k}\}.$$
		If
		$|\widetilde{h}^{k}_x(u)|_{k}=|\pi_{k}^{uc}\widetilde{h}^{k}_x(u)|_{k}$. By (1) of Corollary \ref{C43} and Remark \ref{R1}, we have that
		$$ \delta e^{-2n\varepsilon} \geq|\widetilde{h}^{n}_x(u)|_{n}\geq e^{-2(n-k)\varepsilon} |\widetilde{h}^{k}_x(u)|_{k}.$$
		Thus, $|\widetilde{h}^{k}_x(u)|_{k}\leq \delta e^{-2k\varepsilon}$. If
		$|\widetilde{h}^{k}_x(u)|_{k}=|\pi_{k}^{s}\widetilde{h}^{k}_x(u)|_{k}$. By (2) of Corollary \ref{C43} and Remark \ref{R1}, we have that
		$$|\widetilde{h}^{k}_x(u)|_{k}\leq e^{-2k\varepsilon}|u|_{x} \leq \delta e^{-2k\varepsilon}.$$
		This completes the proof of the lemma.
	\end{proof}
	
	Let $\Gamma_\ell:=\{x\in \Gamma:\ell(x)\leq \ell\}$, and choose a sufficiently large $\ell$ such that $\mu(\Gamma_\ell)>0$. Let $h=f$ or $g$, for every $x\in \Gamma_\ell$, let $\tau_{h}(x)$ be the smallest positive integer $n$ such that $h^n(x)\in \Gamma_\ell$. By the  Poincar\'{e}'s recurrence theorem, we have that  $\tau_{h}(x)<\infty$ for $\mu$-almost every  $x\in \Gamma_\ell$, extend $\tau_{h}(x)$ to $A$ by
	letting $\tau_{h}(x)=0$ if $x\notin \Gamma_{\ell}$.
	For each $\delta>0$,  define a function $\rho_{\delta,h}:A \rightarrow (0,\infty)$ as follows:
	$$\rho_{\delta,h}(x)=\min\{\delta,\delta_1 \ell^{-2}e^{-(\lambda+\varepsilon)\tau_{h}(x)}\}$$
    see \eqref{eq:Expanding} for the choice of $\lambda$.
	Note that $\log \rho_{\delta,h}$ is integrable since $\int \tau_{h}(x) d \mu=(\mu(\Gamma_\ell))^{-1}$, and the family of functions $\{\rho_{\delta,h}:\delta>0\}$ satisfies the conditions of Proposition \ref{P5}.
	Recall that $B(x,r)=\{y\in \cB:|x-y|\leq r\}$, and
	$$B_n(x,\rho_{\delta,h},h)=\bigcap_{k=0}^{n}h^{-k} B(h^{k}(x),\rho_{\delta,h}(h^kx)).$$
	\begin{lemma}\label{L512} Let  $h=f$ or $g$,
		$0<\delta<\delta_1\ell^{-2}$,  $x\in\Gamma_\ell\cap h^{-n}\Gamma_\ell$ and $n\in \N$. For each $k=0,1\cdots,n$, the following properties hold:
		\begin{enumerate}
			\item[(1)]$ \widetilde{h}^{k}_x(\exp^{-1}_x (B_n(x,\rho_{\delta,h},h)\cap B(x,\delta e^{-2n\varepsilon}))) \subset
			\widetilde{B}_{h^k(x)}(\delta \ell e^{-2(n-k)\varepsilon});$
			\item[(2)]$
			\widetilde{h}^{k}_x(\exp^{-1}_x (B_n(x,\rho_{\delta,h},h)\cap h^{-n}(B(h^n(x),\delta e^{-2n\varepsilon})))) \subset
			\widetilde{B}_{h^k(x)}(\delta \ell e^{-2k\varepsilon}).$
		\end{enumerate}	
	\end{lemma}
	\begin{proof}
		For simplicity,
		denote $|\cdot|_{h^{k}(x)}$ by
		$|\cdot|_{k}$.
		For each $y\in B_n(x,\rho_{\delta,h},h)$, let $u=\exp^{-1}_x y$. For every $k\in \{0,1\cdots,n\}$ with $h^k(x)\in \Gamma_\ell$, we have that
		$$|\widetilde{h}^{k}_x(u)|_{k}\leq \ell(h^k(x))|\widetilde{h}^{k}_x(u)|\leq \ell|h^k(x)-h^k(y)| \leq \ell\delta \leq \ell^{-1}\delta_1\leq \delta_1\ell(h^k(x))^{-1}.$$
		If $h^k(x)\notin \Gamma_\ell$, denote by $n_k<k$ the biggest nonnegative integer such that $h^{n_k}(x)\in \Gamma_\ell$, then $\tau_{h}(h^{n_k}(x))>k-n_k$. We have
		$$|\widetilde{h}^{n_k}_x(u)|_{n_k}\leq \ell|h^{n_k}(x)-h^{n_k}(y)|\leq \ell^{-1}\delta_1 e^{-(\lambda+\varepsilon)\tau_{h}(h^{n_k}(x))} \leq \ell^{-1}\delta_1 e^{-(\lambda+\varepsilon)(k-n_k)}.$$
		By (\ref{eq:Expanding}), we have that
		$$ |\widetilde{h}^{k}_x(u)|_{k}\leq e^{\lambda(k-n_k)}|\widetilde{h}^{n_k}_x(u)|_{n_k}\leq \ell^{-1}e^{-(k-n_k)\varepsilon}\delta_1\leq \delta_1\ell(h^k(x))^{-1}$$
        since $\ell(h^{n_k}(x))=\ell$.
		Therefore, we have that
		$$|\widetilde{h}^{k}_x(u)|_{k}\leq \delta_1\ell(h^k(x))^{-1}$$
        for every $k=0,1\cdots,n$.
		Moreover, if $y\in B(x,\delta e^{-2n\varepsilon})$ then $|u|_{x}\leq \ell(x)|u|\leq \ell \delta e^{-2n\varepsilon}$. Similarly, $h^n(y)\in B(h^n(x),\delta e^{-2n\varepsilon})$ implies $|\widetilde{h}^{n}_x(u)|_{n}\leq \ell\delta e^{-2n\varepsilon}$.  Statements (1) and (2) follow immediately from Lemma \ref{L511}.
	\end{proof}

	\begin{lemma}\label{L55}
		For every $0<\delta<\delta_1\ell^{-2} e^{-\lambda-\varepsilon}$ and $x \in \Gamma_\ell \cap f^{-n}\Gamma_\ell \cap (fg)^{-n} \Gamma_\ell$, let
		\begin{align*} \Delta:=&B_n(x,\rho_{\delta,f},f)\cap B(x,\delta e^{-2n\varepsilon}) \\ &\cap f^{-n}(B_n(f^n(x),\rho_{\delta,g},g) )  \cap (fg)^{-n}(B((fg)^n(x), \delta e^{-2n\varepsilon})).
		\end{align*}
		Then we have that $\Delta\subset B_n(x,\delta \ell K_0,fg)$, where $ K_0 $ is the constant  as in Lemma \ref{L42}.
	\end{lemma}
	\begin{proof}
		For each $y\in \Delta$, let $u=\exp^{-1}_x y$. Since $x\in \Gamma_\ell \cap f^{-n}\Gamma_\ell$ and $y\in B_n(x,\rho_{\delta,f},f)\cap B(x,\delta e^{-2n\varepsilon})$, by (1) of Lemma \ref{L512} we have that
   \begin{eqnarray}\label{eq:55}
   \widetilde{f}^{n-k}_{x}u\in \widetilde{B}_{f^{n-k}(x)}(\delta \ell e^{-2k\varepsilon})
   \end{eqnarray}for every $0\leq k\leq n$. Similarly, since  $f^n(x)\in \Gamma_\ell \cap g^{-n}\Gamma_\ell$ and $f^{n}(y)\in  B_n(f^n(x),\rho_{\delta,g},g) )  \cap g^{-n}(B((fg)^n(x), \delta e^{-2n\varepsilon}))$, by (2) of Lemma \ref{L512} we have that
   \begin{eqnarray}\label{add1}
   \widetilde{(fg)}^{k}_{f^{n-k}(x)}\circ   \widetilde{f}^{n-k}_{x}u\in \widetilde{B}_{f^ng^k(x)}(\delta \ell e^{-2k\varepsilon})
   \end{eqnarray} for every $0\leq k\leq n$.

		For each $0\leq k\leq n$, we {\bf claim} that
		\begin{equation}\label{claim2}
		\widetilde{(fg)}^{i}_{f^{n-k}(x)}\circ \widetilde{f}^{n-k}_{x}u \in \widetilde{B}_{(fg)^i(f^{n-k}(x))}(\delta \ell e^{-2\max \{i,k-i\}\varepsilon}) \ \text{for} \, i=0,1,\cdots, k.
		\end{equation}
		We will show the above claim by induction. For $k=0$, the claim follows from (\ref{eq:55}). Suppose that the claim is true for $k-1$, i.e.,
		$$\widetilde{(fg)}^{i-1}_{f^{n-k+1}(x)}\circ \widetilde{f}^{n-k+1}_{x}u \in \widetilde{B}_{(fg)^{i-1}(f^{n-k+1}(x))}(\delta \ell e^{-2\max \{i-1,k-i\}\varepsilon})$$
		for every $1\leq i \leq k$.
		Note that
		$$(\widetilde{g}\circ\widetilde{(fg)}^{i-1})_{f^{n-k+1}x}\circ \widetilde{f}^{n-k+1}_{x} u = \widetilde{(fg)}^{i}_{f^{n-k}(x)}\circ \widetilde{f}^{n-k}_{x}u.$$
		By (\ref{eq:Expanding}) and the  induction assumption, for every $1\leq i \leq k$ we have that
		\begin{align*}
		|\widetilde{(fg)}^{i}\circ \widetilde{f}^{n-k}_{x}u|_{(fg)^{i}(f^{n-k}(x))}
		&\leq e^\lambda|\widetilde{(fg)}^{i-1}_{f^{n-k+1}x}\circ \widetilde{f}^{n-k+1}_{x}u|_{(fg)^{i-1}(f^{n-k+1}(x))}\\
		&\leq e^\lambda \cdot \delta \ell e^{-2\max \{i-1,k-i\}\varepsilon}\\
		&\leq \delta_1 \ell^{-1}e^{-k\varepsilon}.
		\end{align*}
        By the definition of the function $\ell(x)$, we have that
		$$\ell((fg)^{i}(f^{n-k}(x)))=\ell(f^{-(k-i)}g^{i}f^n(x))\leq \ell(f^n(x))e^{(k-i)\varepsilon+i\varepsilon}\leq \ell e^{k\varepsilon} $$
        since $f^n(x)\in \Gamma_\ell$.
		Therefore, we have that
		$$\widetilde{(fg)}^{i}\circ \widetilde{f}^{n-k}_{x}u \in  \widetilde{B}_{(fg)^if^{n-k}(x)}(\delta_1 \ell((fg)^{i}f^{n-k}(x))^{-1})  \quad \forall i=1,2,\cdots,  k.$$
		Applying  Lemma \ref{L511} to vector $\widetilde{f}^{n-k}_{x}u$ and map $\widetilde{fg}$ with \eqref{eq:55} and \eqref{add1}, we have that
		$$\widetilde{(fg)}^{i}_{f^{n-k}(x)}\circ \widetilde{f}^{n-k}_{x}u \in \widetilde{B}_{(fg)^i(f^{n-k}(x))}(\delta \ell e^{-2\max \{k,k-i\}\varepsilon})\quad \forall i=0,1,\cdots, k.$$ This completes the proof of the claim.
		
		By the claim, for every $0 \leq i \leq n$ we have that
		$$|\widetilde{(fg)}^{i}_x u|_{(fg)^i(x)}\leq \delta \ell e^{-2\max \{i,n-i\}\varepsilon}\leq \delta \ell.$$
		Since $ \widetilde{(fg)}^{i}_x u=(fg)^i(y)-(fg)^i(x)$, it follows from  \eqref{eq:norm} that
		\begin{align*}
		|(fg)^i(x)-(fg)^i(y)|=|\widetilde{(fg)}^{i}_x u|\leq K_0 |\widetilde{(fg)}^{i}_x u|_{(fg)^i(x)} \leq  \delta \ell K_0
		\end{align*}
		for every $0 \leq i \leq n$. Hence, $y\in B_n(x,\delta \ell K_0,fg).$  This completes the proof of the lemma.
	\end{proof}
	
\subsection{Proof of  Theorem A}\label{sub5.3}
	
	Let  $N(r)$ denote the minimal number of balls of radius $r$ covering $A$.
	By the assumption (ii) of (H2),  there exist numbers $D_0>0$ and $\delta_2>0$, such that for every $0<r<\delta_2$ we have that
	\begin{equation}\label{eq:Box}
	N(r)\leq r^{-D_0}.
	\end{equation}
	
	\begin{proof}[Proof of  Theorem \ref{A}]
		Take $r\in(0,1)$. Fix $\ell>1$ such that $\mu(\Gamma_\ell)>1-r/5$. Let $\rho_{\delta,f}$ and $\rho_{\delta,g}$ be two families of functions  as in Section \ref{Sub52}.
		
		For $h=f$ or $g$, $n\in \mathbb{N}$ and sufficiently small number $\varepsilon>0$ and $\delta>0$, let
		$$A^{h}_{n,\delta,\varepsilon}=\Big\{x\in \Gamma:\mu(B_k(x,\rho_{\delta,h},h))\geq e^{-k(h_{\mu}(h)+\varepsilon)} \ \forall\ k\geq n\Big\}.$$
		It follows from Proposition \ref{P5} that
		$$h_{\mu}(h)\geq \limsup_{n\rightarrow \infty}  -\frac{1}{n} \log \mu(B_n(x,\rho_{\delta,h},h)) \quad \mu-a.e.$$
		Thus, for every $\delta>0$ we have $\mu(A^{h}_{n,\delta,\varepsilon})\rightarrow 1$ as $n\rightarrow \infty$. Choose $n_h(\delta)>0$, such that $\mu(A^{h}_{n,\delta,\varepsilon})>1-r/5$ for every $n>n_h(\delta)$.
		
		By the definition of $A^{h}_{n,\delta,\varepsilon}$, there are at most $e^{n(h_{\mu}(h)+\varepsilon)}$ disjoint $(n,\rho_{\delta,h},h)$-balls centered at points in $A^{h}_{n,\delta,\varepsilon}$. So the same number of $(n,2\rho_{\delta,h},h)$-balls centered at points in $A^{h}_{n,\delta,\varepsilon}$ can cover $A^{h}_{n,\delta,\varepsilon}$. That is, there exists a set $S_h\subset A^{h}_{n,\delta,\varepsilon}$ with
		$|S_h|\leq e^{n(h_{\mu}(h)+\varepsilon)}$ such that
		$$A^{h}_{n,\delta,\varepsilon} \subset \bigcup_{x\in S_h}B_n(x,2\rho_{\delta,h},h),$$
		where $|S|$ denotes the cardinality of the set $S$.
		
		Take $0<\delta<\min\{\frac{1}{4}\delta_1\ell^{-2} e^{-\lambda-\varepsilon} ,\delta_2\}$. By (\ref{eq:Box}), there exists a set $S_0\subset A$ with $|S_0|=N(2\delta e^{-2n\varepsilon})\leq (2\delta e^{-2n\varepsilon})^{-D_0}$ such that
		$$A\subset \bigcup_{x\in S_0}B(x,2\delta e^{-2n\varepsilon}).$$
		For each $n>\max\{n_f(\delta),n_g(\delta)\}$, let
		$$A_n=A^{f}_{n,\delta,\varepsilon}\cup f^{-n} A^{g}_{n,\delta,\varepsilon} \cap \Gamma_\ell \cap f^{-n}\Gamma_\ell \cap (fg)^{-n} \Gamma_\ell.$$
		Obviously, $\mu(A_n)\geq 1-r$. For every $x_f\in S_f, x_g\in S_g,{x}'\in S_0$ and ${x}''\in S_0$, if the intersection
		\begin{align*} A_n\cap &B_n(x_f,2\rho_{\delta,f},f)\cap B({x}',2\delta e^{-2n\varepsilon}) \\ &\cap f^{-n}(B_n(x_g,2\rho_{\delta,g},g) )  \cap (fg)^{-n}(B({x}'', 2\delta e^{-2n\varepsilon}))
		\end{align*}
		is not empty, then for each $x$ in it, the intersection is contained in the set
		\begin{align*} &B_n(x,4\rho_{\delta,f},f)\cap B(x,4\delta e^{-2n\varepsilon}) \\ &\cap f^{-n}(B_n(f^n(x),4\rho_{\delta,g},g) )  \cap (fg)^{-n}(B((fg)^n(x), 4\delta e^{-2n\varepsilon})).
		\end{align*}
		By  Lemma \ref{L55}, the above set is contained in $B_n(x,4\delta \ell K_0,fg)$.
	There are at most $|S_f|\cdot|S_g|\cdot|S_0|^2 $ different such intersections. Note that $\mu(A_n)>1-r$, these intersections cover $A_n$ and each one is contained in an $(n,4\delta \ell K_0,fg)$-ball. Thus, we have that
		\begin{align*}
		N_n(4\delta \ell K_0,r,fg)&\leq |S_f|\cdot|S_g|\cdot|S_0|^2\\
		&\leq e^{n(h_{\mu}(f)+h_{\mu}(g)+2\varepsilon)}\cdot (2\delta e^{-2n\varepsilon})^{-2D_0}\\
		&\leq (2\delta)^{-2D_0}\cdot e^{n(h_{\mu}(f)+h_{\mu}(g)+(4D_0+2)\varepsilon)}.
		\end{align*}
		Therefore, by Lemma \ref{L52}
		$$h_{\mu}(fg)\leq h_{\mu}(f)+h_{\mu}(g)+(4D_0+2)\varepsilon.$$
		Since $\varepsilon$ is arbitrary, one has that
		$$ h_{\mu}(fg)\leq h_{\mu}(f)+h_{\mu}(g).$$
		This completes the proof of Theorem \ref{A}.
	\end{proof}
	\section{\textbf{Local unstable manifold}}\label{S6}
	This section is a continuation of Section \ref{SEC4}, and the notations $\widetilde{h}_{x}, \widetilde{B}_{x}(r),\delta_1,\varepsilon$ and $\ell(x)$ are the same as in  Section \ref{SEC4}. We will recall the unstable manifold theorem in \cite{Young17}, and give an equivalent characterization of the local unstable manifolds.
	
	Notice that $E^{u}(x,f)$ or $E^{u}(x,g)$ equal to $\{0\}$ if and only if $(f, \mu)$ or $(g, \mu)$ has no positive Lyapunov exponents. In this case, Theorem \ref{B} follows from Ruelle's inequality \cite{Li12} immediately. In the rest of the paper, we  assume that $(f, \mu)$ and $(g, \mu)$ have positive Lyapunov exponents.
	
	For $h=f,g$, $\tau=u,c,s$ and $r>0$, we write $\widetilde{B}^{\tau}_{x}(r,h)=\{v\in E^{\tau}(x,h):|v|_{x}\leq r\}$ and $\widetilde{B}^{cs}_{x}(r,h)=\widetilde{B}^{c}_{x}(r,h)+\widetilde{B}^{s}_{x}(r,h)$. Hence,  $\widetilde{B}_{x}(r)=\widetilde{B}^{u}_{x}(r,h)+\widetilde{B}^{cs}_{x}(r,h)$.
	
	The unstable (stable) manifold theory is well-known for finite dimensional systems.
    See \cite[ Theorem 6.1]{Young17} for this theorem  of maps on a Banach space that we will recall in below.
    See also \cite{Lian11} and \cite{Lian20} for the detailed proofs of this theorem for maps on Hilbert space and maps on Banach space respectively.

	\begin{theorem} \label{Unstable}
		For $h=f$ or $g$, there exists ${\delta}'_1\in(0,\delta_1)$ such that for every $0<\delta<{\delta}'_1$, there exists a unique family of continuous maps $ \{\sigma^{h}_{x}:\widetilde{B}^{u}_{x}(\delta \ell(x)^{-1},h)\rightarrow \widetilde{B}^{cs}_{x}(\delta \ell(x)^{-1},h) \}_{x\in\Gamma}$ so that
		$$\sigma^{h}_{x}(0)=0 \ \text{and} \ \widetilde{h}_x(\graph(\sigma^{h}_{x}))\supset \graph(\sigma^{h}_{h(x)}).  $$
		Moreover,  the  family $\{\sigma^{h}_{x}\}_{x\in \Gamma}$ has the following additional properties with respect to the  norm $|\cdot|_{x}$:
		\begin{enumerate}
			\item[(1)] $ \sigma^{h}_{x} $ is $C^{1+\Lip}$ Fr\'{e}chet differentiable, with $(D\sigma^{h}_{x})_0=0$;
			\item[(2)] $ \Lip\sigma^{h}_{x}\leq 1/10$ and $\Lip(D\sigma^{h}_{x})\leq C\ell(x)$ where $ C > 0 $ is independent of $  x $;
			\item[(3)] if $\widetilde{h}_x(u_i+\sigma^{h}_{x}(u_i))\in \widetilde{B}_{hx}(\delta \ell(hx)^{-1}) $ for $u_i\in \widetilde{B}^{u}_{x}(\delta \ell(x)^{-1},h)$ $(i=1,2)$, then
			\begin{align*}
			|\widetilde{h}_x(u_1+&\sigma^{h}_{x}(u_1))-\widetilde{h}_x(u_2+\sigma^{h}_{x}(u_2)) |_{hx}\\
			\geq &(e^{\lambda^{+}(h)-\varepsilon}-\delta)|u_1+\sigma^{h}_{x}(u_1)-u_2-\sigma^{h}_{x}(u_2)|_{x}.
			\end{align*}
		\end{enumerate}
	\end{theorem}
	For $h=f$ or $g$. Let $W_{\delta}^{u}(x,h)=\graph(\sigma^{h}_{x})$ and define the \textbf{local unstable manifold} at $x$ with respect to $h$ by
	$$w_{\delta}^{u}(x,h)=\exp_x W_{\delta}^{u}(x,h).$$
	The \textbf{global unstable manifold} at $x$ with respect to $h$ is defined as
	$$w^{u}(x,h)=\bigcup_{n\geq 0}h^{n} w_{\delta,x}^{u}(h^{-n}(x),h),$$
	which is an immersed submanifold in $\cB$.
	\begin{remark}\label{add-re1}
		Let $x\in \Gamma$ and $0<\delta<\delta_1'$, one may not have $w_{\delta}^{u}(x,h)\subset A$. However,  by Theorem \ref{Unstable} one has $W_{\delta}^{u}(x,h)\subset \widetilde{h}^n_{h^{-n}x} W_{\delta}^{u}(h^{-n}x,h)$ for each $n>0$. Consequently, one has $w_{\delta}^{u}(x,h)\subset h^n w_{\delta}^{u}(h^{-n}x,h)$ for each $n>0$. Hence, for every $y\in w_{\delta}^{u}(x,h)$ there exists a unique $y_n\in w_{\delta}^{u}(h^{-n}x,h)$ such that $h^n(y_n)=y$.
	\end{remark}
	
	Notice that the next lemma is different to  \cite[Lemma 6.3]{Young17}, here we do not assume that the center subspaces are trivial.
	\begin{lemma} \label{L61}
		For $h=f$ or $g$ and  every $0<\delta<{\delta}'_1$, for every $x\in\Gamma$ we have that
		\begin{align*}
		w_{\delta}^{u}(x,h)=\exp_{x}\Big\{v\in \widetilde{B}_{x}(\delta \ell(x)^{-1}): \ &\forall n \in \N, \exists v_n \in \widetilde{B}_{h^{-n}x}(\delta_1 \ell(h^{-n}x)^{-1}) \ \text{such that} \\
		\widetilde{h}^{n}_{h^{-n}x}v_n=v \ &\text{and} \  |\pi_{h^{-n}(x),h}^{u} v_n|_{h^{-n}(x)}=|v_n|_{h^{-n}(x)}\Big\}
		\end{align*}
		where $\pi_{x,h}^{u}$ is the same as in Corollary \ref{C42}.
		Moreover, if $E^{c}(x,h)=\{0\}$, one has
		\begin{align*}
			w_{\delta}^{u}(x,h)=\exp_{x}\Big\{v\in \widetilde{B}_{x}(\delta \ell(x)^{-1}): \ &\forall n \in \N, \exists v_n \in \widetilde{B}_{h^{-n}x}(\delta_1 \ell(h^{-n}x)^{-1}) \ \text{such that} \\
			&\widetilde{h}^{n}_{h^{-n}x}v_n=v\Big\}.
		\end{align*}
	\end{lemma}
	\begin{proof}
		To simplify the notations, for $\tau=u,cs$ write $\pi^{\tau}_{h^{-k}x,h}$, $|\cdot|_{h^{-k}x}$ as
	$\pi^{\tau}_{-k}$, $|\cdot|_{-k}$ respectively, and let
    \begin{align*}
		R:=\exp_{x}\Big\{v\in \widetilde{B}_{x}(\delta \ell(x)^{-1}): \ &\forall n \in \N, \exists v_n \in \widetilde{B}_{h^{-n}x}(\delta_1 \ell(h^{-n}x)^{-1}) \ \text{such that} \\
		\widetilde{h}^{n}_{h^{-n}x}v_n=v \ &\text{and} \  |\pi_{h^{-n}(x),h}^{u} v_n|_{h^{-n}(x)}=|v_n|_{h^{-n}(x)}\Big\}
	\end{align*}

		For each $y\in w_{\delta}^{u}(x,h)$, let $v=\exp^{-1}_x y\in W_{\delta}^{u}(x,h)$. By Theorem \ref{Unstable},  for every $n\geq 0$ there exists a $v_n\in W_{\delta,h^{-n}x}^{u}(h^{-n}x,h)$ such that $\widetilde{h}^{n}_{h^{-n}x}v_n=v$. Moreover, one can show that  $|\pi_{-n}^{u} v_n|_{-n}=|v_n|_{-n}$ since $\Lip\sigma^{h}_{h^{-n}x}\leq 1/10$. Thus, $y\in R$.
		
		On the other hand, take $v\in \widetilde{B}_{x}(\delta \ell(x)^{-1})$, assume that there exist $\{v_n\}_{n\geq 0}$ such that
		$$v_n \in \widetilde{B}_{h^{-n}x}(\delta_1 \ell(h^{-n}x)^{-1}), \ \widetilde{h}^{n}_{h^{-n}x}v_n=v \ \text{and} \  |\pi_{-n}^{u} v_n|_{-n}=|v_n|_{-n}. $$
		If $v\notin  W_{\delta}^{u}(x,h)$, then there exists  $u\in  W_{\delta}^{u}(x,h)$ such that $\pi^{u}_{x,h}v=\pi^{u}_{x,h}u$ but $\pi^{cs}_{x,h}v\neq\pi^{cs}_{x,h}u$. For each $n>0$, let $u_n=(\widetilde{h}^{n}_{h^{-n}x})^{-1}u \in W^{u}_{\delta}(h^{-n}x, h).$
		Note that
		\begin{align*}
		1\geq \dfrac{|\pi_{-n}^{cs} v_n|_{-n}}{|\pi_{-n}^{u} v_n|_{-n}}\geq \dfrac{|\pi_{-n}^{cs} (v_n-u_n)|_{-n}- |u_n|_{-n}}{|v_n|_{-n}},
		\end{align*}
        $|u-v|_{x}=|\pi^{cs}_{x,h}(u-v)|_{x}$, $|\pi_{-n}^{u} v_n|_{-n}=|v_n|_{-n}$ and
		$|\pi_{-n}^{u} u_n|_{-n}=|u_n|_{-n}$.
		By (2) of Corollary \ref{C43}, we have that
		\begin{align*}
		|u-v|_x \leq e^{2n\varepsilon} |\pi_{-n}^{cs} (v_n-u_n)|_{-n}=e^{2n\varepsilon} |v_n-u_n|_{-n}.
		\end{align*}
		Similarly,  by (1) of Corollary \ref{C43},
		\begin{align*}
		|v|_{x}\geq (e^{4\varepsilon}-\delta_1)^n | v_n|_{-n} \ \text{and} \
		|u|_{x}\geq (e^{4\varepsilon}-\delta_1)^n | u_n|_{-n}.
		\end{align*}
		The above observation yield that
		\begin{align*}
		1&\geq \dfrac{e^{-2n\varepsilon}|u-v|_x-(e^{4\varepsilon}-\delta_1)^{-n}|u|_{x}}{(e^{4\varepsilon}-\delta_1)^{-n}|v|_{x}}\\
		&\geq (\dfrac{e^{4\varepsilon}-\delta_1}{e^{2\varepsilon}})^{n} \dfrac{|u-v|_x}{|v|_{x}}-\frac{|u|_{x}}{|v|_{x}}.
		\end{align*}
		Recall the choice of $\delta_1$ in (\ref{eq:delta}), the right-hand side in the above inequality tends to $\infty$ as $n\rightarrow \infty$. This  contradiction yields that
		 $v\in W_{\delta}^{u}(x,h)$. Hence, we have that $\exp_x v\in w_{\delta}^{u}(x,h).$
		
		To finish the proof, we assume $E^{c}(x,h)=\{0\}$. By Remark \ref{R1}, one has
		$$|u-v|_{x}\leq e^{-2n\varepsilon} |u_n-v_n|_{-n}.$$
		Since $|v_n|_{-n}\leq \delta_1 \ell(h^{-n}x)^{-1}$ and $|u_n|_{-n}\leq \delta_1 \ell(h^{-n}x)^{-1}$, we have
		$$|v_n-u_n|_{-n}\leq 2\delta_1 \ell(h^{-n}x)^{-1}.$$
		Thus, one has that
		$$\frac{1}{n}\log e^{2n\varepsilon}|u-v|_x\leq \frac{1}{n}\log (2\delta_1 \ell(h^{-n}x)^{-1}).$$
		Letting $n\rightarrow \infty$, we have that
		$$2\varepsilon<\limsup_{n\rightarrow \infty} \frac{1}{n}\log \ell(h^{-n}x)^{-1}\leq\limsup_{n\rightarrow \infty} \frac{1}{n}\log (e^{n\varepsilon}\ell(x)^{-1})=\varepsilon.$$
		This contradiction implies that $v\in W^{u}_{\delta}(x,h)$. This completes the proof of the lemma.
	\end{proof}

    \section{\textbf{Proof of Theorem \ref{B}}}\label{S7}
     In this section, the main aim is to  give  conditions under which the inequality in Theorem \ref{A} becomes an equality. We first show the local unstable manifolds of the two commuting maps are equal under some suitable conditions. Moreover,  if the center subspaces of the two commuting maps are trivial, the equality of the metric entropies is established.

	\subsection{Proof of the first statement}This section will prove the first statement of Theorem \ref{B}, i.e.,  if $E^{u}(x,f)=E^{u}(x,g)$ for $\mu$-almost every $x$,
	then $w_{\delta}^{u}(x,f)=w_{\delta}^{u}(x,g)$ for  $\mu$-almost every $x$ provided that $\delta$ is sufficiently  small. Without loss of generality,  assume that $E^{u}(x,f)=E^{u}(x,g)$ for every $x\in \Gamma$. By \eqref{eq:split}, one also has that  $E^{cs}(x,f)=E^{cs}(x,g)$ for every $x\in \Gamma$.
	
	Recall the following well-known result about graph transformations. Let $0<\delta<{\delta}'_1$, $x\in\Gamma$ and
	$$\cW(x,h)=\Big\{\sigma:\widetilde{B}^{u}_{x}(\delta \ell(x)^{-1},h)\rightarrow \widetilde{B}^{cs}_{x}(\delta \ell(x)^{-1},h) \, | \,  \sigma(0)=0, \ \Lip \sigma \leq \frac{1}{10}\Big\}.$$
	where $h=f$ or $g$.
    For $\sigma \in \cW(x,h)$, define a map
	$\Psi_{x,h}(\sigma): \widetilde{B}^{u}_{hx}(\delta \ell(hx)^{-1},h)\rightarrow \widetilde{B}^{cs}_{hx}(\delta \ell(hx)^{-1},h)$  with the following property
	$$\widetilde{h}_{x}(\graph(\sigma))\supset \graph(\Psi_{x,h}(\sigma)).$$
    The map $\Psi_{x,h}(\sigma)$ is called the graph transform of $\sigma$ with respect to $h$, and it is well-defined in the following sense.

	\begin{lemma}[\cite{Young17} Lemma 6.2] \label{L62}
		Let $h=f$ or $g$ and $x\in \Gamma$, then the following properties hold:
		\begin{enumerate}
	    \item[(1)] for every $\sigma\in \cW(x,h)$,  $\Psi_{x,h}(\sigma)$ exists and belongs to $\cW(hx,h)$;
	    \item[(2)] there exists a constant $c\in (0,1)$ such that
	    $$|||\Psi_{x,h}(\sigma_1)-\Psi_{x,h}(\sigma_2) |||_{hx}\leq c|||\sigma_1-\sigma_2|||_{x}\quad \forall\,\sigma_1, \sigma_2 \in \cW(x,h) $$
	    where
	    $\displaystyle{|||\sigma|||_x=\sup_{v\in \widetilde{B}^{u}_{x}(\delta \ell(x)^{-1},h)\setminus \{0\}} \dfrac{|\sigma(v)|_{x}}{|v|_{x}} }$.
	    \end{enumerate}
	\end{lemma}
	
    Recall the choice of $\lambda$ in \eqref{eq:Expanding} and ${\delta}'_1$ in Theorem \ref{Unstable}, we have the following result.
	\begin{lemma}\label{L63}
		 Let ${\delta}'_2=e^{-\lambda-\varepsilon}{\delta}'_1$. If $E^{u}(x,f)=E^{u}(x,g)$ for every $x\in \Gamma$, then for every $0<\delta<{\delta}'_2$
		we have that
		$$\widetilde{f}_x(W_{\delta}^{u}(x,g))\supset W_{\delta}^{u}(fx,g), \    \widetilde{g}_x(W_{\delta}^{u}(x,f))\supset   W_{\delta}^{u}(gx,f) \quad \forall\, x\in \Gamma.$$
	\end{lemma}
	\begin{proof}
		Fix $x\in \Gamma$, since $E^{u}(x,f)=E^{u}(x,g)$, we have that $E^{cs}(x,f)=E^{cs}(x,g)$ by \eqref{eq:split}. Hence, if there is no confusion caused, we will simply write $E^{\tau}(x,h)$ as $E^{\tau}(x)$ for $\tau=u,cs$ and $h=f,g$. This also implies that $\cW(x,f)=\cW(x,g)$, and  we denote the common set by $\cW(x)$.
		
		In the following, we shall show that
		$\widetilde{f}_x(W_{\delta}^{u}(x,g))\supset W_{\delta}^{u}(fx,g) $, the other one can be proven in a similar fashion.
		
		Given $\sigma_{x}^{g}\in \cW(x)$  with $\graph(\sigma_{x}^{g})=W_{\delta}^{u}(x,g)$ as in Theorem \ref{Unstable}. By Lemma \ref{L62},
		there exists  $\tilde{\sigma}_{fx}\in  \cW(fx)$ such that
		\begin{equation}\label{eq:63}
		\widetilde{f}_x W_{\delta}^{u}(x,g)\supset \graph(\tilde{\sigma}_{fx}).
		\end{equation}
		It suffices to prove that $\tilde{\sigma}_{fx}=\sigma^{g}_{fx}$. For every $v \in  \tilde{B}^{u}_{fx}(\delta \ell(fx)^{-1})$, let $u:=v+\tilde{\sigma}_{fx} v\in  \graph(\tilde{\sigma}_{fx}).$ By (\ref{eq:63}), there exists  ${u}_0' \in W_{\delta}^{u}(x,g)$ such that  $\widetilde{f}_x ({u}'_0)=u$. Since ${u}'_0 \in W_{\delta}^{u}(x,g)$,
		 there exists  ${u}'_n\in W_{\delta}^{u}(g^{-n}x,g)$ such that $\widetilde{g}^n_{g^{-n}x} ({u}'_n)={u}'_0$ and  $|\pi_{g^{-n}(x),f}^{u} {u}'_n|_{g^{-n}(x)}=|{u}'_n|_{g^{-n}(x)}$  for every $n\geq 0$ . Let $u_{n}=\widetilde{f}_{g^{-n}x} {u}'_n$, since ${u}'_n\in \tilde{B}_{g^{-n}x}(\delta \ell(g^{-n}x)^{-1})$,
		by \eqref{exp-1step} and (1) of Corollary \ref{C43}   we have that
		$$u_{n}\in \tilde{B}_{g^{-n}(fx)}(\delta_1 \ell(g^{-n}fx)^{-1}),\ |\pi_{g^{-n}(fx),f}^{u} u_n|_{g^{-n}(fx)}=|u_n|_{g^{-n}(fx)}.$$
		By construction of $u_n$, we have that $\widetilde{g}^n_{g^{-n}(fx)} u_n=u$ for every $n\geq 1$. It follows from Lemma \ref{L61} that
		$u\in W_{\delta}^{u}(fx,g).$
		This yields  that $\tilde{\sigma}_{fx}= \sigma^{g}_{fx}$. This completes the proof of the lemma.
	\end{proof}
	Now we will prove the first statement of Theorem \ref{B}, we write it as the following theorem.
	\begin{theorem}\label{T6}
		Suppose $E^{u}(x,f)=E^{u}(x,g)$ for every $x\in \Gamma$. For every $0<\delta<{\delta}'_2$ ( $ {\delta}'_2 $ is the same as in Lemma \ref{L63}),
		we have that $w^{u}_{\delta}(x,f)=w^{u}_{\delta}(x,g)$ for every $x\in \Gamma$.
	\end{theorem}
	\begin{proof}
		Given $x\in\Gamma$ and $u\in W^{u}_{\delta}(x,f)$, it follows from Lemma \ref{L63} that
		 $$\widetilde{g}^{n}_{g^{-n}x}(W_{\delta}^{u}(g^{-n}x,f))\supset W_{\delta}^{u}(x,f)\ \ \forall \, n\geq 0. $$
		Thus, there exists  $u_n\in W^{u}_{\delta}(g^{-n}x,f)$ such that $\widetilde{g}^{n}_{g^{-n}x} u_n=u$ for every $n\geq 0$.  Since $u_n\in W^{u}_{\delta}(g^{-n}x,f)\subset \tilde{B}_{g^{-n}x}(\delta \ell(g^{-n}x)^{-1})$ and $E^{u}(g^{-n}x,f)=E^{u}(g^{-n}x,g)$, we have that
		$$|\pi_{g^{-n}(x),f}^{u} u_n|_{g^{-n}x}=|\pi_{g^{-n}(x),g}^{u} u_n|_{g^{-n}x}=|u_n|_{g^{-n}x}.$$
		It follows from Lemma \ref{L61} that $u\in  W^{u}_{\delta}(x,g).$ Thus, we have $W^{u}_{\delta}(x,f) \subset  W^{u}_{\delta}(x,g).$
		
		Similarly, one can show that $W^{u}_{\delta}(x,g) \subset  W^{u}_{\delta}(x,f).$ Therefore, $W^{u}_{\delta}(x,g)=  W^{u}_{\delta}(x,f)$. So we conclude that  $ w^{u}_{\delta}(x,f)=w^{u}_{\delta}(x,g) $ for every $x\in \Gamma$.
	\end{proof}
	
	\subsection{Proof of the second statement}\label{tc} In this section, we will show the equality of the measure theoretic entropies under the additional condition {that the center subspaces of $f$ and $g$ are trivial.}

	\subsubsection{Unstable stacks}\label{S71}
	In this part, we assume that  $E^{u}(x,f)=E^{u}(x,g)$ for every $x\in \Gamma$. As before, we will simply write  $E^{\tau}(x,h)$ as $E^{\tau}(x)$ for $\tau=u,cs$ and $h=f,g$ if there is no confusion caused. By Theorem \ref{T6}, we have $\sigma^f_{x}|_{\widetilde{B}_x^{u}(r\ell(x)^{-1})}=\sigma^g_{x}|_{\widetilde{B}_x^{u}(r\ell(x)^{-1})}$
	for small $r>0$. Therefore, $w_r^{u}(x,f)=w_r^{u}(x,g)$ and we will simply write these notations as $w_r^{u}(x)$ and $\sigma_{x}$ respectively.
	
	In order to introduce some results on the regularity of the unstable manifolds, we will introduce some notations as follows.  For $\tau=u,cs$, let $B_x^{\tau}(r)=\{v\in E^{\tau}(x):|v|\leq r\}$.  Denote by
	$(C(B_{x}^{u}(\delta \ell_0^{-3}),E^{cs}(x)),||\cdot||)$ the space of all continuous functions from $B_{x}^{u}(\delta \ell_0^{-3})$ to $E^{cs}(x)$ with the $C^0$ super norm $||\cdot||$.
	
	Since $x\mapsto E^{\tau}(x)$ is $\mu$-continuous, by Definition \ref{D1} there exists a sequence of increasing compact subsets $\{K_n\}$ so that $\mu(\bigcup_n K_n)=1$ and the map $x\mapsto E^{\tau}(x)$ is continuous on $K_n$ for every $n\ge 1$. Recall that  $\Gamma_{\ell}=\{x\in \Gamma: \ell(x)\leq \ell\}$, where $\ell(x)$ is defined in Lemma \ref{L43}, and the constants $ {\delta}'_2 $ and  $K_0$  are given  in Lemma \ref{L63} and Lemma \ref{L42} respectively.

	\begin{theorem}[\cite{Young17} Lemma 6.5]\label{P7}
     Fix $\ell_0$ and $n_0>1$  so that $\mu(\Gamma_{\ell_0}\cap K_{n_0})>0$, and take $x_0\in \Gamma_{\ell_0}\cap K_{n_0}.$ For $r>0$ and $x\in \Gamma,$  let
		$$U(x,r):=\Gamma_{\ell_0}\cap K_{n_0}\cap B(x,r).$$
		 Then, there exists  $0<\delta_3<{\delta}'_2$ such that for every $0<\delta\leq \delta_3$, the following properties hold for sufficiently small  $r_0>0$:
		\begin{enumerate}
			\item[(1)] for every $y\in U(x_0,r_0)$, there exists a continuous function $\sigma_{x_0}^{y}: Dom(y) \rightarrow E^{cs}(x_0)$ such that
			$$\exp_{x_0}\graph(\sigma_{x_0}^{y})=\exp_{y}\graph(\sigma_{y}|_{B^{u}_{y}(2\delta \ell_{0}^{-3})})$$
            where $Dom(y) \supset B_{x_0}^u(\delta \ell_0^{-3})$ is a subset of $E^{u}(x_0)$;
			\item[(2)] the mapping $\Theta:U(x_0,r_0)\rightarrow C(B_{x_0}^{u}(\delta \ell_0^{-3}),E^{cs}(x_0))$ defined by $\Theta(y)=\sigma_{x_0}^{y}|_{B_{x_0}^{u}(\delta \ell_0^{-3})}$ is continuous.
		\end{enumerate}
	\end{theorem}

    Let $x_0,U(x_0,r_0)$ and $\Theta$  be given as above, and let $\bar{U}\subset U(x_0,r_0)$ be a compact subset. The following set
	$$\cS=\bigcup_{y\in \bar{U}} \exp_{x_0}(\graph \Theta(y)),$$
	is called  a \textbf{stack of local unstable manifolds}, and the set $\exp_{x_0}(\graph \Theta(y))$ is called the \textbf{unstable leave}.
	We refer the reader to Sections 6.2 and 7.2 in \cite{Young17} for more details about unstable stack.
	
    \subsubsection{Construction of the partition }\label{sub7.2}
    In this section, we will construct a measurable partition $\eta$ of a full $\mu$-measure set such that $h_\mu(h,\eta)=h_\mu(h)$ for $h=f$ or $g$. See \cite{Rohlin67} for the detailed description of the quantity $h_\mu(h,\eta)$.

	We first recall some notations. For a measurable partition $\eta$, let $\eta(x)$ denote the element of $\eta$ containing $ x $. For two measurable partitions $\eta_1$ and $\eta_2$, we say $\eta_2$ is a refinement of $\eta_1$
    if $\eta_1(x)\supset \eta_2(x)$ for $\mu$-almost every $x$,  and denote it by $\eta_1\leq \eta_2 $. Let further that  $\eta_1 \vee \eta_2=\{A\cap B:A\in \eta_1, B\in\eta_2\}$, and $h^{-1} \eta_1=\{h^{-1}(A):A\in\eta_1\}$ where $h=f$ or $g$. A measurable partition $\eta$ is called
	$h$-decreasing if $\eta \leq h^{-1} \eta$, and $\eta$ is called $(f,g)$-decreasing, if $\eta$ are both $f$-decreasing and $g$-decreasing.

    Fix $0<\delta< \delta_3$, let
	$$\cS=\bigcup_{y\in \bar{U}} \exp_{x_0}(\graph \Theta(y)|_{\delta \ell_0^{-3}})$$
	where $x_0\in \Gamma$ is a fixed point so that $\mu(\cS)>0$. Denote by $\widetilde{\cS}$ the set $\bigcup_{n,k\geq0}f^ng^k \cS$, then $\mu(\widetilde{\cS})=1$ since $\mu$ is $(f,g)$-ergodic.
	
	Note that the distinct unstable leaves in $\cS$ do not intersect with each other.  Let $\xi$ be the measurable partition of $\cS$ into unstable leaves (see \cite[Lemma 6.7]{Young17}). For $n,k\geq0$, let $\xi_{n,k}=\{f^{n}g^k(W):W\in \xi\}\cup (\widetilde{\cS}-f^ng^k(\cS))$ be a measurable partition of $\widetilde{\cS}$. Finally, consider the measurable partition of $\widetilde{\cS}$ as follows:
	$$\eta=\bigvee_{n,k\geq0}\xi_{n,k}.$$	
	By the construction of $\eta$, one can show that $y\in \eta(x)$ if and only if
	\begin{equation}\label{eq:eta}
		\left\{
		\begin{alignedat}{2}
			&f^{-n}g^{-k}y \in \xi(f^{-n}g^{-k}x) \quad &\text{if}& \ f^{-n}g^{-k}(x)\in \cS; \\
			&y\notin f^{n}g^{k}(\cS),                                    &\text{if}& \ f^{-n}g^{-k}(x)\notin \cS
		\end{alignedat}
		\right.
	\end{equation}
	for every $n,k\geq0$. 
	Hence, we have  $\eta$ is $(f,g)$-decreasing.
	As Proposition 8.1 in \cite{Hu96}, by choosing small $\delta>0$ appropriately, one can make the partition $\eta$ subordinate to the unstable foliation $W^{u}$ (see definition in \cite[Definition 7.2]{Young17}). 
	
\subsubsection{Proof of the second statement}
    In this section, we will first show that $h_\mu(f)=h_\mu(f,\eta)$ and $h_\mu(g)=h_\mu(g,\eta)$ with respect to the partition constructed in the previous section. By a standard argument, the second statement of Theorem \ref{B} follows immediately.
	\begin{lemma}\label{entropy}
		Suppose that $E^{c}(x,f)=E^{c}(x,g)=\{0\}$ for every $x\in \Gamma$. 
		Let $\eta$ be a partition constructed as above. Then
		$$h_{\mu}(f,\eta)=h_{\mu}(f), \ h_{\mu}(g,\eta)=h_{\mu}(g).$$
	\end{lemma}
	\begin{proof}
	We only prove that $h_{\mu}(f,\eta)=h_{\mu}(f)$, and $h_{\mu}(g,\eta)=h_{\mu}(g)$ can be proven in a similar fashion.
	
	Let $\mu=\int \mu_{e} d \hat{\mu}(e)$ be the ergodic decomposition of $\mu$ with respect to $f$.
	It suffices to show that $h_{\mu_e}(f)=h_{\mu_e}(f,\eta)$ for $\hat{\mu}$-almost every $e$.
	
	Since $\mu(\widetilde{\cS})=1$, for $\hat{\mu}$-almost every $e$ one has 
	$$\mu_{e}(\widetilde{\cS})=\mu_{e}(\bigcup_{n,k\geq0} f^{n}g^{k}\cS)=1.$$
	For each $\mu_{e}$ with $\mu_{e}(\widetilde{\cS})=1$. Since $\mu_{e}$ is $f$-ergodic, we can take $k_e\in \N$ such that $\mu_{e}(g^{k_e}\bigcup_{n\geq 0}f^{n}\cS)=1$.
	Denote by $\eta_{e}:=\bigvee_{n\geq0}\xi_{n,k_e}$ and $\cS_e:=g^{k_e}\bigcup_{n\geq 0}f^{n}\cS$. 
	By Poincar\'{e}'s recurrence theorem, one can show that $\vee_{n \geq 0} f^{-n}\eta_{e}$ is a partitions of $\cS_e$ into points $\mu_e$-modulo $0$. 
	Using the same argument in the proof of Lemma 7.15 in \cite{Young16}, one can show that 
	$h_{\mu_e}(f)=h_{\mu_e}(f,\eta_{e})$.
	
	We now prove $h_{\mu_e}(f,\eta)=h_{\mu_e}(f,\eta_{e})$. Note that $\eta \geq \eta_{e}$, $\eta_{e}$ is $f$-decreasing,
	$\eta(x)\subset w^u(x)$, and the diameter of $(f^{-n} \eta_e)(x)$ tends to $0$ as $n\rightarrow +\infty$.
    By the same argument in the proof of Lemma 2.11 in \cite{Zang22} (also in \cite[Lemma 3.1.2]{Young85}), we have 
    $$h_{\mu_{e}}(f,\eta)=h_{\mu_{e}}(f,\eta\vee \eta_{e})=h_{\mu_{e}}(f,\eta_{e}).$$
    Therefore, we have $h_{\mu_e}(f,\eta)=h_{\mu_e}(f,\eta_{e})$.
    This completes the proof.
	\end{proof}

    Finally,  we will prove the second statement of Theorem \ref{B}, we write it as the following theorem.
	\begin{theorem}
		Let $\cB$ be a separable Banach space with norm $|\cdot|$  and $A\subset \cB$ a compact subset. Assume that $f,g:\cB \rightarrow \cB$ satisfy the conditions (H1)-(H3), $\mu \in \cE(f,g,A)$ satisfies the condition (H4),  
		$E^{u}(x,f)=E^{u}(x,g)$ and $E^{c}(x,f)=\{0\}=E^{c}(x,g)$ for $\mu$-almost every $x$. Then, we have that
		$$h_{\mu}(f)+h_{\mu}(g)=h_{\mu}(fg).$$
	\end{theorem}
	\begin{proof}
		Let $\eta$ be the partition constructed in Section \ref{sub7.2}, then (recall that $\eta$ is $ (f,g) $-decreasing) we have that 
		\begin{align*}
		h_{\mu}(fg)&\geq h_{\mu}(fg,\eta)\\
        &= H_{\mu}(f^{-1}g^{-1}\eta|\eta)\\
		&=H_{\mu}(f^{-1}g^{-1}\eta \vee g^{-1}\eta |\eta)\\
		&=H_{\mu}(g^{-1}\eta |\eta)+H_{\mu}(f^{-1}g^{-1}\eta | g^{-1}\eta)\\
		&=h_{\mu}(g,\eta)+h_{\mu}(f,\eta)\\
        &=h_{\mu}(f)+h_{\mu}(g).
		\end{align*}
		By Theorem \ref{A}, we have that $ h_{\mu}(f)+h_{\mu}(g)=h_{\mu}(fg).$
	\end{proof}
	
	
    \bibliography{LZ22-sub}
    \bibliographystyle{plain}
\end{document}